\documentclass[smallextended]{svjour3}

\smartqed

\usepackage{graphicx}
\usepackage{psfrag}
\usepackage{caption}
\usepackage{subfig}
\usepackage{url}
\usepackage{color}
\usepackage{cite}
\usepackage{epsfig}
\usepackage{amsmath,amssymb}

\providecommand{\rset}[1]{\mathbb{R}^}
\providecommand{\abs}[1]{\lvert#1\rvert}
\providecommand{\norm}[1]{\lVert#1\rVert}
\newcommand{\lu}{\mathcal{L}}
\newcommand{\calK}{\mathcal{K}}
\newcommand{\ein}{\epsilon_{\text{in}}}

\journalname{Mathematical Programming Computation}

\begin{document}

\title{DuQuad: an inexact (augmented) dual first order algorithm
for quadratic programming}

\author{ Ion Necoara \and  Andrei Patrascu}

\institute{I. Necoara and A. Patrascu \at
Automatic Control and Systems Engineering Department, University Politehnica Bucharest, 060042 Bucharest, Romania\\
Tel.: +40-21-9195, Fax: +40-21-9195\\
\email{ion.necoara@acse.pub.ro}
}

\date{Received: April 2015 / Accepted: date}

\maketitle

\begin{abstract}
In this paper we present the solver DuQuad  specialized for solving
general convex quadratic problems arising in many engineering
applications. When it is difficult to project on the primal feasible
set, we use the (augmented) Lagrangian relaxation to handle the complicated
constraints and then, we apply dual first order algorithms based on
inexact dual gradient information for solving the corresponding dual
problem. The iteration complexity analysis is based on two types of
approximate primal solutions:  the primal last iterate and an
average of primal iterates. We provide computational complexity
estimates on the primal suboptimality and feasibility violation of
the generated approximate primal solutions. Then, these algorithms
are implemented in the programming language C in DuQuad, and
optimized for low iteration complexity and low memory footprint.
DuQuad has a dynamic Matlab interface which make the process of
testing, comparing, and analyzing the algorithms simple. The
algorithms are implemented using only basic arithmetic and logical
operations and are suitable to run on low cost hardware. It is shown
that if an approximate solution is sufficient for a given
application, there exists problems where some of the implemented
algorithms obtain the solution faster than state-of-the-art
commercial solvers.

\keywords{Convex quadratic programming \and (augmented) dual relaxation \and
first  order algorithms, \and rate of convergence \and arithmetic
complexity.}
\end{abstract}

%%%%%%%%%%%%%%%%%%%%%%%%%%%%%%%%%%%%%%%%%%%%%%%

\section{Introduction}
\noindent Nowadays, many engineering applications can be  posed as
convex  quadratic  problems (QP).  Several important applications
that can be modeled in this framework such us model predictive
control for a dynamical linear system
\cite{DomZgr:12,NecNed:13,NedNec:12,PatBem:12,StaSzu:14} and its
dual called  moving horizon estimation \cite{FraVuk:14}, DC optimal
power flow problem for a power system \cite{ZimMur:11}, linear
inverse problems arising in many branches of science
\cite{BecTeb:09,WanLin:13}  or  network utility maximization
problems \cite{WeiOzd:10} have attracted great attention lately.
Since the computational power has increased by many orders in the
last decade, highly efficient and reliable  numerical optimization
algorithms  have been developed for solving the optimization
problems arising from these applications in very short time. For
example, these hardware and numerical recent advances made it
possible to solve linear predictive control problems of nontrivial
sizes within the range of microseconds and  even on  hardware
platforms with limited computational power and memory
\cite{JerLin:12}.

\noindent  The theoretical foundation of quadratic programming dates
back to the work by Frank \& Wolfe \cite{FraWol:56}. After the
publication of the paper \cite{FraWol:56} many numerical algorithms
have been  developed in the literature that exploit efficiently the
structure arising in this class of problems. Basically,  we can
identify three popular classes of algorithms to solve quadratic
programs: active set methods, interior point methods and (dual)
first order methods.

\noindent \textit{Active set methods} are based on the observation
that  quadratic problems with equality constraints are equivalent to
solving a linear system. Thus,  the iterations in these methods are
based on solving a linear system and updating the active set (the
term active set refers to the subset of constraints that are
satisfied as equalities by the current estimate of the solution).
Active set general purpose solvers are adequate for small-to-medium
scale quadratic problems, since the numerical complexity per
iteration  is cubic in the dimension of the problem.  Matlab's
\textit{quadprog} function implements a primal active set method.
Dual active set methods are available in the codes
\cite{BarBie:06,FerKir:14}.

\noindent \textit{Interior point methods} remove the inequality
constraints  from the problem formulation using a barrier term  in
the objective function for penalizing the constraint violations.
Usually a logarithmic barrier terms is used and the resulting
equality constrained nonlinear convex problem is solved by the
Newton method. Since the iteration complexity grows also cubically
with the dimension, interior-point solvers are also the standard for
small-to-medium scale QPs. However,  structure exploiting interior
point solvers have been also  developed for particular  large-scale
applications: e.g. several solvers exploit the sparse structure of
the quadratic problem arising in predictive control (CVXGEN
\cite{MatBoy:09}, FORCES \cite{DomZgr:12}). A parallel interior
point code that exploits special structures in the Hessian of
large-scale structured  quadratic programs have been developed
in~\cite{GonGro:07}.

\noindent \textit{First order methods} use only gradient information
at  each iterate by computing a step towards the solution of the
unconstrained problem and then projecting this step onto the
feasible set. Augmented Lagrangian algorithms for solving general nonconvex
problems are presented  in the software package Lancelot \cite{ConGou:92}.
For convex QPs with simple constraints  we can use primal first order
methods for solving the quadratic program as in \cite{Ulm:11}. In this case the main computational effort per iteration  consists of a matrix-vector product. When the projection on the primal feasible set is hard to compute, an alternative to
primal first order methods is to use the Lagrangian relaxation to
handle the complicated constraints and then to apply dual first
order algorithms for solving the dual. The computational complexity
certification of first order methods for solving the (augmented)
Lagrangian dual of general  convex problems is studied e.g. in
\cite{LanMon:08,NecNed:13,NedNec:12,NecFer:15,NedOzd:09} and of
quadratic  problems is studied in \cite{Gis:14,PatBem:12,StaSzu:14}.
In these methods the main computational effort consists of solving
at each iteration a Lagrangian QP problem  with simple constraints
for a given multiplier, which allows us to determine the value of
the dual gradient for that multiplier,  and then update the dual
variables using matrix-vector products.   For example, the toolbox
FiOrdOs \cite{Ulm:11} auto-generates code for primal or dual fast
gradient methods  as proposed in \cite{RicMor:11}. The  algorithm in
\cite{PatBem:12} dualizes only the inequality constraints of the QP
and assumes available a solver for linear systems  that is able to
solve the Lagrangian inner problem. However,  both implementations
\cite{Ulm:11,PatBem:12} do not consider the important aspect  that
the Lagrangian inner problem cannot be solved exactly in practice.
The effect of inexact computations in dual gradient values on the
convergence of dual first order methods  has been analyzed in detail
in \cite{NecNed:13,NedNec:12,NecFer:15}. Moreover,  most of these
papers generate approximate primal solutions through averaging
\cite{NecNed:13,NedNec:12,NedOzd:09}. On the other hand, in practice
usually the last primal  iterate is employed, since in practice
these methods converge faster in the primal last iterate than in a primal average
sequence. These issues motivate our work~here.

\noindent \textit{Contributions}. In this paper we analyze the
computational complexity of several (augmented) dual first order methods
implemented in DuQuad for solving convex quadratic problems.
Contrary to most of the results from the literature
\cite{NedOzd:09,PatBem:12}, our approach allows us to use inexact
dual gradient information (i.e. it allows to solve the (augmented)
Lagrangian inner problem approximately) and therefore is able to tackle more
general quadratic convex problems and to solve practical
applications. Another important feature of our approach is that we
provide also complexity results for the primal latest iterate, while
in much of the previous literature convergence rates in an average
of primal iterates are given \cite{NecNed:13,NedNec:12,NedOzd:09,PatBem:12}.
We derive in a unified framework the computational complexity of the dual and augmented dual (fast) gradient methods in terms of primal suboptimality and feasibility violation using inexact dual gradients and two types of approximate primal solutions: the last primal iterate and  an average of  primal iterates. From our knowledge this paper
is the first where both approaches, dual and augmented dual first order methods,
are analyzed  uniformly.  These algorithms are also implemented in the efficient programming language C in DuQuad, and optimized for low iteration
complexity and low memory footprint. The toolbox has a dynamic Matlab
interface which make the process of testing, comparing, and analyzing the algorithms simple. The algorithms are implemented using only basic arithmetic and logical
operations and thus are suitable to run on low cost hardware. The main
computational bottleneck in the methods implemented in DuQuad is the
matrix-vector product. Therefore, this toolbox can be used for
solving either QPs on hardware with limited resources or sparse QPs with
large dimension.

\noindent  \textit{Contents}. The paper is organized as follows. In
section \ref{sec_pf} we describe the optimization problem that we
solve in DuQuad. In Section \ref{sec_duquad} we describe the the
main theoretical aspects that DuQuad is based on, while in Section
\ref{numerical_tests} we present some numerical results obtained
with DuQuad.

\noindent \textit{Notation}. For $x,y \in \rset^n$ denote the scalar
product by $\langle x,y \rangle = x^T y$ and the  Euclidean norm by
$\|x\|=\sqrt{x^T x}$. Further, $[u]_X$ denotes the projection of $u$
onto convex set $X$ and $\text{dist}_{X}(u) =\|u -[u]_X\|$ its
distance. For a matrix $G \in \rset^{m \times n}$ we use the
notation $\|G\|$ for the spectral norm.

%%%%%%%%%%%%%%%%%%%%%%%%%%%%%%%%%%%%%%%%%%%%%%%%%%%%%%%%%%%%%%%%%%%%%55
%%%%%%%%%%%%%%%%%%%%%%%%%%%%%%%%%%%%%%%%%%%%%%%%%%%%%%%%%%%%%%%%%%%%%%

\section{Problem formulation}
\label{sec_pf} \noindent In DuQuad  we consider a general convex
quadratic problem (QP) in the form:
\begin{align}\label{original_primal}
F^* = & \min_{u \in U} F(u) \quad \left(=
\frac{1}{2} u^T Q u + q^T u\right) \\
& \text{s.t:} \;\;\; G u + g \in \calK, \nonumber
\end{align}
where $F:\rset^n \rightarrow \rset^{}$ is a convex quadratic
function  with the Hessian $Q \succeq 0$, $G \in \rset^{p
\times n}$, $U \subseteq \rset^n$ is a simple compact  convex set,
i.e. a box $U= [\text{lb} \; \text{ub}]$,   and ${\calK}$ is
either the cone $\rset^p_{-}$ or the cone $\{0\}$. Note that our
formulation allows to incorporate in the QP either linear inequality
constraints $\calK=\rset^p_{-}$ (arising e.g. in sparse formulation
of predictive control and network utility maximization) or linear
equality constraints $\calK=\{0\}$ (arising  e.g. in condensed
formulation of predictive control and DC optimal power flow). In
fact the user can define linear constraints of the form:
$\bar{\text{lb}} \leq \bar{G} u+ \bar{g} \leq
\bar{\text{ub}}$ and depending on the values for
$\bar{\text{lb}}$ and $\bar{\text{ub}}$ we have linear
inequalities or equalities.  Throughout the paper  we assume that there
exists a finite optimal Lagrange multiplier $\lambda^*$ for the QP
\eqref{original_primal} and it is difficult to project on the
feasible set of problem~\eqref{original_primal}: $${\cal X}_{QP} =
\{u \in U: \; G u+ g \in \calK\}.$$ Therefore,
solving the primal problem \eqref{original_primal} approximately
with primal first order methods  is numerically difficult and thus
we usually use (augmented) dual first order  methods for finding an approximate
solution for \eqref{original_primal}.  By moving the complicating
constraints $G u+ g \in \calK$ into the cost via
Lagrange multipliers we define the (augmented) dual function:
\begin{align}
\label{dual_fc} d_\rho(\lambda) = \min_{u \in U} \mathcal{L}_\rho(u,\lambda),
\end{align}
where $\mathcal{L}_\rho(u,\lambda)$ denotes the (augmented)
Lagrangian w.r.t.~the complicating constraints $G u +g \in
\calK$, i.e.:
\begin{equation}
\label{auglag}
 \mathcal{L}_\rho(u,\lambda)  = \min\limits_{s \in \calK} \; F(u) +
 \langle \lambda, G u + g - s \rangle + \frac{\rho}{2}\norm{Gu + g - s}^2
\end{equation}
where the regularization parameter  $\rho \geq 0$. We denote
$s(u,\lambda) = \arg\min\limits_{s \in \calK} \langle \lambda, G u + g -s \rangle + \frac{\rho}{2}\norm{Gu+g - s}^2$ and
observe that:
\begin{equation*}
 s(u,\lambda) = \begin{cases} \left[ Gu+g + \frac{1}{\rho}\lambda \right]_{\calK} & \text{if} \;\; \rho>0 \\
                 0 &\text{if} \;\; \rho=0.
                \end{cases}
\end{equation*}
Using this observation in the formulation \eqref{auglag}, we obtain:
\begin{align}
\label{auglag1}
 \mathcal{L}_{\rho}(u, \lambda) = \begin{cases}
  F(u) +  \frac{\rho}{2} \text{dist}_{\mathcal{K}} \left( Gu+g + \frac{1}{\rho} \lambda \right)^2 - \frac{1}{2\rho}\norm{\lambda}^2,
  & \text{if} \; \rho > 0\\
  F(u) +  \langle \lambda, G u + g \rangle,
  & \text{if}\; \rho = 0.
\end{cases}
\end{align}
In order to tackle general convex quadratic programs, in DuQuad we consider the following two options:

\noindent \textbf{Case 1}:  if $Q \succ 0$, i.e. $Q$ has the smallest eigenvalue
$\lambda_{\min}(Q) >0$, then we consider $\rho=0$ and recover the ordinary Lagrangian function.

\noindent \textbf{Case 2}:  if  $Q \succeq 0$, i.e.
$Q$ has the smallest eigenvalue $\lambda_{\min}(Q) =0$,
then we consider $\rho >0$ and recover the augmented Lagrangian function.

\noindent Our formulation of the (augmented) Lagrangian
\eqref{auglag1}  and the previous two cases allow us to thereat in a
unified framework both approaches, dual and augmented dual first
order methods, for general convex QPs.   We  denote by $u(\lambda)$
the optimal solution  of the \textit{inner problem} with simple
constraints $u \in U$:
\begin{align}
\label{ul} u(\lambda)  = \arg \min_{u \in U}
\mathcal{L}_\rho(u,\lambda).
\end{align}
Note that  for both cases described above the
(augmented) dual function is differentiable everywhere. Moreover,  the gradient
of the (augmented) dual function $d_\rho(\lambda)$ is $L_{\text{d}}$-Lipschitz
continuous  and given by \cite{NecNed:13,NedNec:12,Nes:04,Roc:76}:
\begin{equation}
\label{Ld}
\nabla d_{\rho} (\lambda) = Gu(\lambda) + g - s(u(\lambda),\lambda) \;\; \text{and} \;\; L_{\text{d}}= \frac{\|G\|^2}{\lambda_{\min}(Q) +
\rho\|G\|^2}
\end{equation}
for all $\lambda \in \rset^p$. Since the dual function has  Lipschitz continuous  gradient, we can
derive bounds on $d_\rho$ in terms of a linear and a quadratic model
(the  so-called \textit{descent lemma}) \cite{Nes:04}:
\begin{align}
\label{descentlemma} 0 \leq   [d_\rho(\mu) +
\langle {\nabla} d_\rho(\mu),\lambda - \mu \rangle] - d_\rho(\lambda) \leq
\frac{L_{\text{d}}}{2} \|\mu - \lambda\|^2  \quad \forall \lambda,
\mu \in  \rset^p.
\end{align}
Descent lemma is essential in proving convergence rate for first
order methods \cite{Nes:04}.    Since we assume the existence of a finite optimal Lagrange multiplier for \eqref{original_primal}, strong duality holds and thus the \textit{outer problem} is smooth and satisfies:
\begin{align}
\label{dual_pr} F^* = \max_{\lambda \in \mathcal{K}_d} d_\rho(\lambda),
\end{align}
where
$$\mathcal{K}_d = \begin{cases} \rset^p_+, & \text{if} \;\; Q \succ 0 \;\; \text{and} \;\; \mathcal{K} = \rset^p_{-} \\
\rset^p, & \;\; \text{otherwise}. \end{cases}$$ Note that, in
general,  the smooth (augmented) dual problem \eqref{dual_pr} is not
a QP,  but has simple constraints. We denote a primal optimal
solution by $u^*$ and a dual optimal solution by  $\lambda^*$. We
introduce $\Lambda^* \subseteq \calK_d$ as the set of optimal
solutions of the smooth dual problem \eqref{dual_pr} and define  for
some $\lambda^0 \in \rset^p$ the following finite quantity:
\begin{equation}
\label{eq_multipleirs_bounded} \mathcal{R}_{\text{d}} =
\min \limits_{\lambda^* \in \Lambda^*} \|\lambda^* - \lambda^0\|.
\end{equation}
In the next section we
present a general first order algorithm for convex optimization with
simple constraints that is used frequently in our toolbox.

%%%%%%%%%%%%%%%%%%%%%%%%%%%%%%%%%%%%%%%%%%%%%%%%%%%%%%%%%%%%%%%%%%%%

\subsection{First order methods}
\noindent In this section we present a framework for  first order
methods generating an approximate  solution for a smooth convex
problem with simple constraints in the form:
\begin{equation}
\label{aux_prob}
\phi^* = \min_{x \in X} \; \phi(x),
\end{equation}
where $\phi: \rset^n \to \rset^{}$ is a convex function   and $X$ is
a simple convex set (i.e. the projection on this set is easy).
Additionally, we assume that $\phi$ has  Lipschitz continuous
gradient  with constant $L_{\phi} >0$ and is strongly convex with
constant $\sigma_{\phi} \geq 0$.     This general framework covers
important particular algorithms \cite{Nes:04,BecTeb:09}: e.g.
gradient algorithm,  fast gradient algorithm for smooth problems, or
fast gradient algorithm for problems with smooth and strongly convex
objective function. Thus, we will analyze the iteration complexity
of  the following general  first order method that updates two
sequences $(x^k, y^k)_{k \geq 0}$  as follows:
\begin{center}
\framebox{
\parbox{7cm}{
\begin{center}
\textbf{ Algorithm {\bf FOM ($\phi,X$)} }
\end{center}
{Given $x^0 = y^1 \in X$, for $k\geq 1$ compute:}
\begin{enumerate}
\item ${x}^{k}= \left[ y^k - \frac{1}{L_{\phi}} \nabla \phi(y^k)\right]_X$,
\item $y^{k+1} = x^k + \beta_k (x^k - x^{k-1})$,
\end{enumerate}
}}
\end{center}
where $\beta_k$ is the parameter  of the method and we  choose it in
an appropriate way depending on the properties of function $\phi$.
More precisely, $\beta_k$ can be updated as follows:

\noindent \textbf{GM}: in the Gradient Method $\beta_k=
\frac{\theta_k -1 }{\theta_{k+1}}$, where $\theta_k=1$ for all $k$.
This is equivalent with  $\beta_k=0$ for all $k$. In this case
$y^{k+1} = x^k$ and thus we have the classical gradient update:
${x}^{k+1}= [ x^k - \frac{1}{L_{\phi}} \nabla \phi(x^k)]_X$.

\noindent \textbf{FGM}: in the Fast Gradient Method for smooth
convex problems $\beta_k=\frac{\theta_k -1 }{\theta_{k+1}}$, where
$\theta_{k+1} = \frac{1 + \sqrt{1 + 4 \theta_k^2 }}{2}$ and
$\theta_1=1$. In this case we get a particular version of Nesterov's
accelerated scheme \cite{Nes:04} that updates two sequences
$(x^{k},y^k)$ and has been analyzed in detail in \cite{BecTeb:09}.

\noindent \textbf{FGM}$_\sigma$: in fast gradient algorithm for
smooth convex problems with strongly convex objective function, with
constant $\sigma_{\phi} > 0$, we choose
$\beta_k=\frac{\sqrt{L_{\phi}} -
\sqrt{\sigma_{\phi}}}{\sqrt{L_{\phi}} + \sqrt{\sigma_{\phi}}}$ for
all $k$. In this case we get a particular version of Nesterov's
accelerated scheme \cite{Nes:04} that  also updates two sequences
$(x^{k},y^k)$.

\noindent The  convergence rate of Algorithm {\bf FOM}($\phi,X$) in
terms of function values is given in the next lemma:

\begin{lemma} \cite{BecTeb:09,Nes:04}
\label{lemma_sublin_dfg} For smooth convex problem \eqref{aux_prob}
assume that the objective  function  $\phi$ is strongly convex with
constant $\sigma_\phi \geq 0$ and  has Lipschitz continuous gradient
with constant $L_{\phi} >0$. Then,  the sequences $\left(x^k,
y^k\right)_{k\geq 0}$ generated by Algorithm {\bf FOM}($\phi,X$)
satisfy:
\begin{equation}
\label{bound_gen_dfg1} \phi(x^k) - \phi^* \!\leq\!  \min
\!\left(\!\! \left( \!1 \!-\!
\sqrt{\frac{\sigma_\phi}{L_\phi}}\right)^{k-1} \!\!\!\!\!\! L_{\phi}
\mathcal{R}^2_{\phi}, \frac{2 L_{\phi}
\mathcal{R}^2_{\phi}}{(k\!+\!1)^{p(\beta_k)}} \!\right)
\end{equation}
where $\mathcal{R}_{\phi}  = \min\limits_{x^*_{\phi} \in X^*_{\phi}}  \| x^0 - x^*_{\phi}\|$,
with  $X^*_{\phi}$  the optimal set of \eqref{aux_prob}, and $p(\beta_k)$ is defined as follows:
\begin{align}
\label{pbetak}
p(\beta_k) =
\begin{cases} 1 & \mbox{if } \beta_k = 0 \\
2 & \mbox{otherwise}.
\end{cases}
\end{align}
\end{lemma}
Thus, Algorithm {\bf FOM} has linear convergence provided that
$\sigma_\phi>0$. Otherwise, it has sublinear convergence.

%%%%%%%%%%%%%%%%%%%%%%%%%%%%%%%%%%%%%%%%%%%%%%%%%%%%%%%%%%%%%%%%%%%%%
%%%%%%%%%%%%%%%%%%%%%%%%%%%%%%%%%%%%%%%%%%%%%%%%%%%%%%%%%%%%%%%%%

\section{Inexact (augmented) dual first order methods}
\label{sec_duquad} \noindent In this section we describe an inexact
dual (augmented) first order framework implemented in DuQuad,
a solver able to find an approximate solution for the quadratic program
\eqref{original_primal}. For a given accuracy $\epsilon >0$,
$u_\epsilon \in U$ is called an $\epsilon$-\textit{primal
solution} for problem \eqref{original_primal} if the following
inequalities hold:
\[ \text{dist}_{\calK} (G u_\epsilon + g) \leq {\cal O}(\epsilon)
\quad \text{and} \quad  |F(u_\epsilon) - F^*| \leq {\cal
O}(\epsilon).
\]

\noindent The main function in  DuQuad is the one implementing the
general Algorithm {\bf FOM}. Note that if the feasible set ${\cal
X}_{QP}$ of \eqref{original_primal} is simple, then we can call
directly  {\bf FOM}($F,{\cal X}_{QP}$) in order to obtain an
approximate solution for  \eqref{original_primal}. However, in
general the projection on ${\cal X}_{QP}$ is as difficult as solving
the original problem. In this case we resort to the (augmented) dual formulation
\eqref{dual_pr} for finding an $\epsilon$-primal solution for the
original  QP \eqref{original_primal}. The main idea in DuQuad is
based on the following observation:  from \eqref{ul}--\eqref{Ld} we
observe that for computing the gradient value of the dual function
in some multiplier $\lambda$, we need to solve exactly the inner
problem \eqref{ul}; despite the fact that, in some cases, the (augmented) Lagrangian
$\lu_\rho(u,\lambda)$ is quadratic and the feasible set $U$ is
simple in \eqref{ul}, this inner problem generally cannot be solved exactly.
Therefore, the main iteration in DuQuad consists of
two steps:

\vspace{0.2cm}

\noindent \textbf{Step 1}: for a given inner accuracy $\ein$ and a
multiplier $\mu \in \rset^p$ solve approximately  the inner problem
\eqref{ul} with accuracy $\ein$ to obtain an approximate solution
$\bar{u}(\mu)$ instead of the exact solution $u(\mu)$, i.e.:
\begin{align}
\label{duquad_in} 0 \le \lu_\rho(\bar{u}(\mu),\mu) - d_\rho(\mu) \leq \ein.
\end{align}
In DuQuad, we obtain an approximate solution $\bar{u}(\mu)$ using the
Algorithm {\bf FOM}($\lu_\rho(\cdot,\mu),U$). From, Lemma
\ref{lemma_sublin_dfg} we can estimate tightly the number of
iterations that we need to perform in order to get an
$\ein$-solution  $\bar{u}(\mu)$ for \eqref{ul}: the Lipschitz
constant is $L_{\lu} = \lambda_{\max} (Q) + \rho \|G\|^2$, the
strong convexity constant is $\sigma_{\lu} = \lambda_{\min} (Q +
\rho G^T G)$ (provided that e.g. $\calK=\{0\}$) and
${\cal R}_{\lu} \leq D_{U}$ (the diameter of
the box set $U$). Then, the number of iterations that we need to
perform for computing  $\bar{u}(\mu)$ satisfying \eqref{duquad_in}
can be obtained from  \eqref{bound_gen_dfg1}.

\vspace{0.2cm}

\noindent \textbf{Step 2}: Once an $\ein$-solution  $\bar{u}(\mu)$
for \eqref{ul} was found, we update at the outer stage the Lagrange
multipliers using again Algorithm {\bf FOM}($d_\rho,{\calK}_d$).
Note that for updating the Lagrange multipliers we use instead of
the true value of the dual gradient $\nabla d_\rho(\mu) = G
u(\mu) + g - s(u(\mu),\mu)$,
an approximate value given by: $\bar{\nabla} d_\rho(\mu) =
G \bar{u}(\mu) + g - s(\bar{u}(\mu),\mu)$.

\noindent In \cite{DevGli:14,NedNec:12,NecNed:13,NecFer:15} it has
been proved separately, for dual and augmented dual first order
methods, that using an appropriate value for $\ein$ (depending on
the desired accuracy $\epsilon$ that we want to solve the QP
\eqref{original_primal}) we can still preserve the convergence rates
of  Algorithm {\bf FOM}($d_\rho,{\calK}_d$) given in Lemma
\ref{lemma_sublin_dfg}, although we use inexact dual gradients. In
the sequel, we derive in a unified framework the computational
complexity of the dual and augmented dual (fast) gradient methods.
From our knowledge, this  is the first time when both approaches,
dual and augmented dual first order methods, are analyzed uniformly.
First, we  show that by introducing  inexact values for the dual
function and for its gradient given by the following expressions:
\begin{equation}\label{inexact_framework}
\bar{d}_\rho(\mu) = \lu_\rho(\bar{u}(\mu),\mu) \;\; \text{and} \;\;
 \bar{\nabla} d_\rho(\mu) = G \bar{u}(\mu) + g - s(\bar{u}(\mu),\mu),
\end{equation}
then we have a similar descent relation as in \eqref{descentlemma}
given in the  following lemma:

\begin{lemma}
Given $\ein > 0$ such that \eqref{duquad_in} holds,
then based on the definitions \eqref{inexact_framework}
we get the following inequalities:
\begin{align}
\label{descentlemmainexact}  0 \leq [\bar{d}_\rho(\mu)  \!+ \langle
\bar{\nabla} d_\rho(\mu),\lambda - \mu \rangle] - d_\rho(\lambda)
\leq\! L_{\text{d}} \|\mu -  \lambda\|^2 +\! 2 \ein  \quad \forall
\lambda,\mu \!\in\! \rset^p.
\end{align}
\end{lemma}

\begin{proof}
From the definition of $d_{\rho}$, \eqref{duquad_in} and
\eqref{inexact_framework}    it can be derived:
\begin{align*}
d_{\rho} (\lambda) & = \min\limits_{u \in U, \; s \in \calK}
F(u) + \langle \lambda, Gu+g -s \rangle + \frac{\rho}{2}\norm{Gu+g-s}^2\\
&\le F(\bar{u}(\mu)) + \left\langle \lambda, G \bar{u}(\mu) + g - s(\bar{u}(\mu),\mu) \right\rangle  + \frac{\rho}{2}\left \| G \bar{u}(\mu) + g - s(\bar{u}(\mu),\mu) \right \|^2 \\
& = \mathcal{L}_{\rho}(\bar{u}(\mu),\mu) + \langle G\bar{u}(\mu) + g - s(\bar{u}(\mu),\mu), \lambda - \mu \rangle \\
& = \bar{d}_{\rho}(\mu) + \langle \bar{\nabla} d_{\rho}(\mu), \lambda - \mu\rangle,
\end{align*}
which proves the first inequality. In order to prove the second inequality,
let $\tilde{u} \in U$ be a fixed primal point such that
$\mathcal{L}_{\rho}(\tilde{u},\mu) \ge d_{\rho}(\mu)$. Then, we note that
the nonnegative function $h(\mu) = \mathcal{L}_{\rho}(\tilde{u},\mu) -d_{\rho}(\mu) \ge 0$
has Lipschitz gradient with constant $L_d$ and  thus we have \cite{Nes:04}:
\begin{align*}
\frac{1}{2 L_d} & \left\| \left(G \tilde{u} + g - s(\tilde{u},\mu) \right) - \nabla d_{\rho}(\mu)\right\|^2  =  \frac{1}{2 L_d}  \norm{\nabla h(\mu)}^2 \\
& \le  h(\mu) - \min_{\nu \in \rset^p} h(\nu) \le \mathcal{L}_{\rho}(\tilde{u},\mu) - d_{\rho}(\mu).
\end{align*}
Taking now $\tilde{u} = \bar{u}(\mu)$ and using \eqref{duquad_in}, then we obtain:
\begin{equation}\label{inexact_gradient_rel}
   \norm{\bar{\nabla} d_{\rho} (\mu) - \nabla d_{\rho}(\mu)} \le \sqrt{2 L_d \ein}  \qquad \forall \mu \in \rset^p.
\end{equation}
Furthermore, combining \eqref{inexact_gradient_rel} with
\eqref{descentlemma} and  \eqref{duquad_in} we have:
\begin{align*}
& d_{\rho}(\lambda)
\ge \bar{d}_{\rho}(\mu) + \langle \nabla d_{\rho}(\mu), \lambda -\mu \rangle -
 \frac{L_d}{2}\norm{\lambda - \mu}^2 - \ein \\
& \ge \bar{d}_{\rho}(\mu) + \langle \bar{\nabla} d_{\rho}(\mu), \lambda - \mu \rangle
- \frac{L_d}{2}\norm{\lambda - \mu}^2 + \langle \nabla d_{\rho} (\mu) - \bar{\nabla} d_{\rho}(\mu), \lambda - \mu \rangle
- \ein\\
& \ge \bar{d}_{\rho}(\mu) \!+ \! \langle \bar{\nabla} d_{\rho}(\mu), \lambda - \mu \rangle
- \frac{L_d}{2}\norm{\lambda - \mu}^2 - \norm{\bar{\nabla} d_{\rho}(\mu) \!-\! \nabla d_{\rho}(\mu)} \norm{\lambda - \mu} - \ein\\
& \overset{\eqref{inexact_gradient_rel}}{\ge} \bar{d}_{\rho}(\mu) + \langle \bar{\nabla} d_{\rho} (\mu), \lambda - \mu \rangle
 - \frac{L_d}{2}\norm{\lambda - \mu}^2  - \sqrt{2L_d \ein}\norm{\lambda -\mu} - \ein.
\end{align*}
Using the relation $\sqrt{ab} \le (a + b)/2$ we have:
\begin{equation*}
 d_{\rho}(\lambda) \ge \bar{d}_{\rho}(\mu) + \langle \bar{\nabla} d_{\rho} (\mu), \lambda - \mu \rangle
 - L_d\norm{\lambda - \mu}^2  - 2 \ein,
\end{equation*}
which shows  the second inequality of our lemma.  \qed
\end{proof}

\noindent   This lemma will play a major role in proving rate of
convergence  for the methods presented in this paper. Note that in
\eqref{descentlemmainexact} $\ein$ enters linearly, while in
\cite{DevGli:14,NedNec:12} $\ein$ enters quadratically in the
context of augmented Lagrangian and thus in the sequel we will get
better convergence estimates than those in the previous papers. In
conclusion, for solving the dual problem \eqref{dual_pr} in DuQuad
we use the following inexact (augmented) dual first order
algorithm:
\begin{center}
\framebox{
\parbox{7.7cm}{
\begin{center}
\textbf{ Algorithm {\bf DFOM ($d_\rho,\calK_d$)} }
\end{center}
{Given $\lambda^0 = \mu^1 \in \calK_d$, for $k\geq 1$ compute:}
\begin{enumerate}
\item $\bar{u}^k$ satisfying \eqref{duquad_in} for $\mu = \mu^k$,
i.e. $\bar{u}^k = \bar{u}(\mu^k)$
\item ${\lambda}^{k}= \left[ \mu^k + \frac{1}{2L_{\text{d}}}
\bar{\nabla} d_\rho(\mu^k)\right]_{\calK_d}$,
\item $\mu^{k+1} = \lambda^k + \beta_k (\lambda^k - \lambda^{k-1})$.
\end{enumerate}
}}
\end{center}
Recall that  $\bar{u}^k = \bar{u}(\mu^k)$ satisfying the inner criterion \eqref{duquad_in} and  $\bar{\nabla} d_\rho(\mu^k) = G \bar{u}^k + g - s(\bar{u}^k,\mu^k)$. Moreover,
$\beta_k$ is chosen as follows:
\begin{itemize}
\item  \textbf{DGM}: in (augmented) Dual Gradient  Method $\beta_k=
\frac{\theta_k -1 }{\theta_{k+1}}$, where $\theta_k=1$ for all $k$,
or equivalently   $\beta_k=0$ for all~$k$, i.e. the ordinary gradient algorithm.

\item \textbf{DFGM}: in (augmented) Dual Fast Gradient Method
$\beta_k=\frac{\theta_k -1 }{\theta_{k+1}}$, where $\theta_{k+1} =
\frac{1 + \sqrt{1 + 4 \theta_k^2 }}{2}$ and $\theta_1=1$, i.e. a variant of Nesterov's accelerated scheme.
\end{itemize}

\noindent Therefore, in DuQuad we can solve the smooth (augmented) dual problem
\eqref{dual_pr} either with dual gradient method \textbf{DGM}
($\beta_k =0$) or with dual fast gradient method \textbf{DFGM}
($\beta_k$ is updated based on $\theta_k$).    Recall that for
computing $\bar{u}^k$ in DuQuad we use Algorithm {\bf
FOM}($\lu_\rho(\cdot,\mu^k),U$) (see the discussion of Step 1).
When applied to inner subproblem \eqref{ul}, Algorithm {\bf
FOM}($\lu_\rho(\cdot,\mu^k),U$)  will converge linearly provided
that $\sigma_{\lu} > 0$. Moreover,  when
applying Algorithm {\bf FOM}($\lu_\rho(\cdot,\mu^k),U$) we use
warm start: i.e. we start our iteration from previous computed
$\bar{u}^{k-1}$. Combining the inexact descent relation
\eqref{descentlemmainexact} with Lemma \ref{lemma_sublin_dfg} we
obtain the following convergence rate for the general Algorithm {\bf
DFOM}($d_\rho,\calK_d$) in terms of dual function values of
\eqref{dual_pr}:

\begin{theorem} \cite{DevGli:14,NedNec:12,NecNed:13}
\label{lemma_sublin_dfg_inexact} For the smooth (augmented) dual  problem
\eqref{dual_pr} the dual  sequences $\left(\lambda^k, \mu^k\right)_{k\geq
0}$  generated by Algorithm {\bf DFOM}($d_\rho,\calK_d$) satisfy the following convergence estimate on dual suboptimality:
\begin{equation}
\label{bound_gen_dfg} F^* - d_\rho(\lambda^k)  \leq   \frac{4
L_{\text{d}} \mathcal{R}^2_{\text{d}}}{(k+1)^{p(\beta_k)}} +
2(k+1)^{p(\beta_k)-1} \ein,
\end{equation}
where recall $\mathcal{R}_{\text{d}} =
\min \limits_{\lambda^* \in \Lambda^*} \|\lambda^* - \lambda^0\|$
and $p(\beta_k)$ is  defined as in \eqref{pbetak}. \qed
\end{theorem}

\noindent Note that in \cite[Theorem 2]{DevGli:14}, the convergence rate of \textbf{DGM} scheme is provided  in the average dual iterate $\hat{\lambda}^k = \frac{1}{k+1}
\sum\limits_{j=0}^k \lambda^j$  and not in the last dual iterate $\lambda^k$. However, for a uniform treatment  in Theorem \ref{lemma_sublin_dfg_inexact} we redefine the dual  final point (the dual last iterate $\lambda^k$ when some stopping criterion is satisfied)
as follows: $\lambda^k = \left[\hat{\lambda}^k + \frac{1}{2L_{\text{d}}} \bar{\nabla} d_\rho(\hat{\lambda}^k)\right]_{\calK_d}$.

%%%%%%%%%%%%%%%%%%%%%%%%%%%%%%%%%%%%%%%%%%%%

\subsection{How to choose inner accuracy $\ein$ in DuQuad}
\noindent We now show how to choose the inner accuracy $\ein$ in
DuQuad.  From Theorem \ref{lemma_sublin_dfg_inexact} we conclude
that in order to get $\epsilon$-dual suboptimality, i.e. $ F^* -
d_\rho(\lambda^k) \leq \epsilon$, the inner accuracy $\ein$ and the number of outer iteration $k_{\text{out}}$ (i.e. number of updates of Lagrange multipliers) have to
be chosen as follows:
\begin{equation}
\label{choose_ein}
\ein =
\begin{cases} \frac{\epsilon}{4} & \mbox{if } \;\; \textbf{DGM} \\
 \frac{\epsilon\sqrt{\epsilon}}{8 \mathcal{R}_{\text{d}} \sqrt{2L_{\text{d}}}}
 & \mbox{if } \;\; \textbf{DFGM},
\end{cases}
\hspace{20pt}
k_{\text{out}} = \begin{cases}
 \frac{8 L_d \mathcal{R}_d^2}{\epsilon} & \text{if} \;\; \textbf{DGM} \\
 \sqrt{\frac{8 L_d \mathcal{R}_d^2}{\epsilon}} & \text{if} \;\; \textbf{DFGM}.
\end{cases}
\end{equation}
Indeed, by enforcing each term of the right hand side of
\eqref{bound_gen_dfg} to be smaller than $\frac{\epsilon}{2}$
we obtain first the bound on the number of the outer iterations $k_{\text{out}}$.
By replacing this bound into the expression of $\ein$, we also obtain how to choose $\ein$, i.e  the estimates \eqref{choose_ein}.
We conclude that the inner QP \eqref{ul} has to be solved with
higher accuracy in dual fast gradient algorithm \textbf{DFGM} than
in dual  gradient algorithm \textbf{DGM}. This shows that dual
gradient algorithm  \textbf{DGM} is  robust to inexact (augmented) dual first order information, while dual fast gradient algorithm \textbf{DFGM} is sensitive  to
inexact computations (see also Fig. \ref{dfgm_sensitivity}).  In DuQuad
the user can choose either Algorithm \textbf{DFGM} or Algorithm
\textbf{DGM} for solving the (augmented) dual problem \eqref{dual_pr} and he can
also choose the inner accuracy $\ein$ for solving the inner problem
(in the toolbox the default values for $\ein$ are taken of the same order as in \eqref{choose_ein}).
\begin{figure}[h!]
\centering
\includegraphics[width=0.52\textwidth,height=4.5cm]{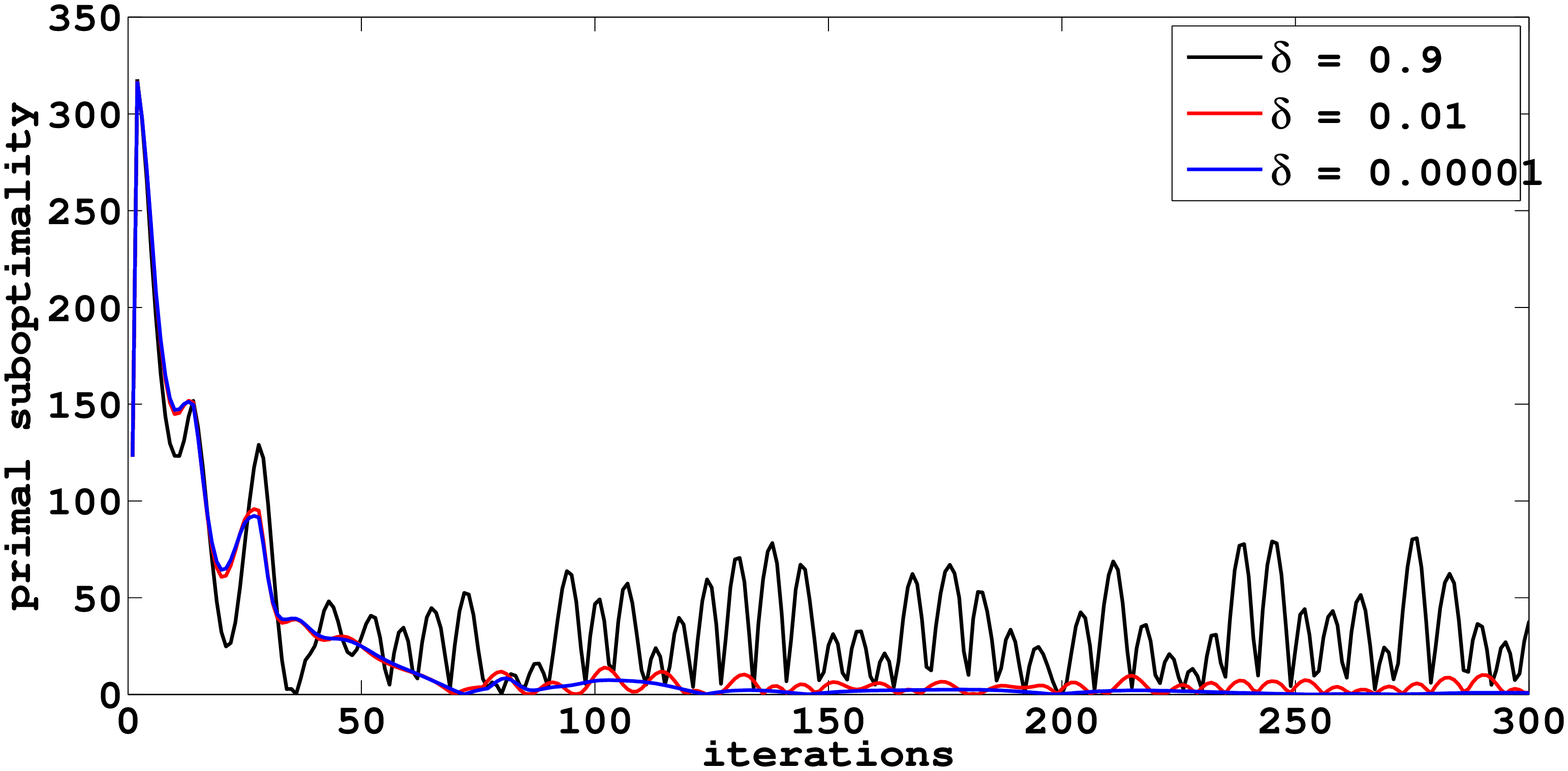}
\hspace*{-0.7cm}
\includegraphics[width=0.52\textwidth,height=4.5cm]{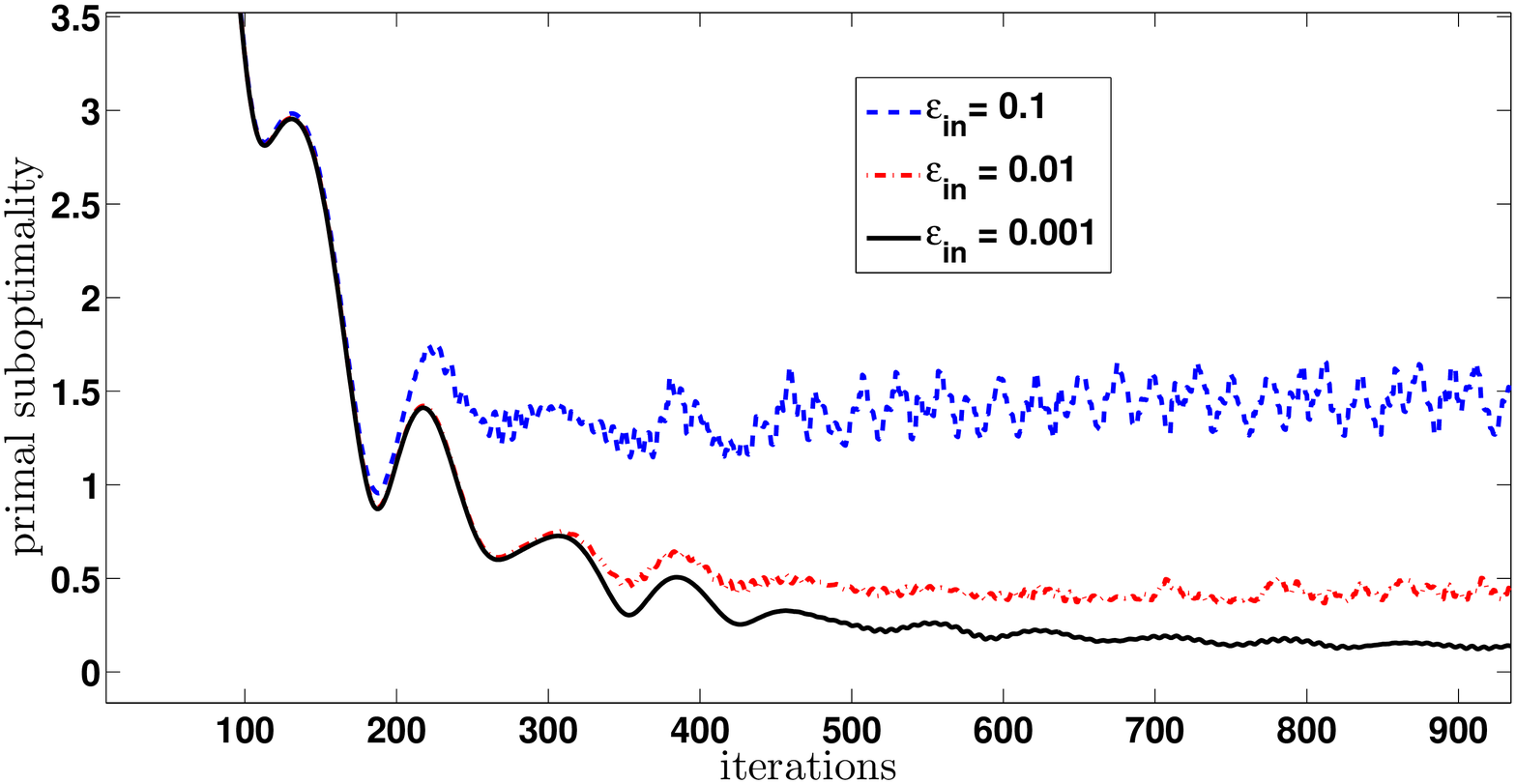}
\caption{Behavior of Algorithms \textbf{DGM} (left),  \textbf{DFGM}(right) in terms of primal suboptimality  w.r.t. inner accuracy $\ein$ for a strongly convex  QP  with $\epsilon=0.01$.}
\label{dfgm_sensitivity}
\end{figure}

%%%%%%%%%%%%%%%%%%%%%%%%%%%%%%%%%%%%%%%%%%%%%%%%%%%%%%%%%%%%%%%%%%%%%%%%%%%%
%%%%%%%%%%%%%%%%%%%%%%%%%%%%%%%%%%%%%%%%%%%%%%%%%%%%%%%%%%%%%%%%%%%%%%%%%%%%%

\section{How to recover an $\epsilon$-primal solution in DuQuad}
\noindent It is natural to  investigate  how to recover an
$\epsilon$-primal solution for the original QP \eqref{original_primal}.
Since dual suboptimality is given in the last dual iterate
$\lambda^k$, it is natural to consider as an approximate
primal solution the last primal iterate generated by Algorithm {\bf DFOM}
in $\lambda^k$, i.e.:
\begin{equation}\label{last_primal}
\bar{u}_\epsilon^k = \bar{u}(\lambda^k).
\end{equation}
Note that the last primal iterate $\bar{u}_\epsilon^k =
\bar{u}(\lambda^k)$ coincides with $\bar{u}^k = \bar{u}(\mu^k)$  for
Algorithm \textbf{DGM}. However, for Algorithm \textbf{DFGM} these
two sequences are different, i.e. $\bar{u}_\epsilon^k \not =
\bar{u}^k$.   We will show below that the last primal iterate
$\bar{u}_\epsilon^k$ is an  $\sqrt{\epsilon}$-primal solution for
the original QP \eqref{original_primal}, provided that $ F^* -
d_\rho(\lambda^k) \leq \epsilon$. We can also construct an
approximate primal solution based on an average of all previous
primal iterates generated by Algorithm {\bf DFOM} , i.e.:
\begin{equation}\label{average_primal}
 \hat{u}_\epsilon^k = \sum_{j=1}^k   \frac{\theta_j \bar{u}^j}{S_k},
 \quad S_k=\sum_{j=1}^k \theta_j.
\end{equation}
Recall that $\theta_j=1$ in Algorithm \textbf{DGM} and $\theta_j$ is
updated according to the rule $\theta_{j+1} = \frac{1 + \sqrt{1 + 4
\theta_j^2 }}{2}$ and $\theta_1=1$ in Algorithm  \textbf{DFGM}. In the sequel, we
prove that the average of primal iterates
sequence   $\hat{u}_\epsilon^k$ is an  $\epsilon$-primal solution
for the original QP  \eqref{original_primal},  provided that $ F^* -
d_\rho(\lambda^k) \leq \epsilon$.

\noindent Before proving  primal rate of convergence for Algorithm
\textbf{DFOM} we  derive a bound on  $\norm{\lambda^{k+1} -
\lambda^*}$, with $\lambda^k$ generated by algorithm \textbf{DFOM},
bound that  will be used in the proofs of our convergence results.
In the case of \textbf{DGM}, using its particular iteration, for any
$\lambda \in \mathcal{K}_d$, we have:
\begin{align}
  \norm{\lambda^{k+1}-\lambda}^2 &=  \norm{\lambda^k  - \lambda}^2  + 2\langle \lambda^{k+1} - \lambda^k, \lambda^k - \lambda \rangle + \norm{\lambda^{k+1} - \lambda^k}^2 \nonumber\\
&  = \norm{\lambda^k  - \lambda}^2  + 2\langle \lambda^{k+1} - \lambda^k, \lambda^{k+1}- \lambda \rangle - \norm{\lambda^{k+1} - \lambda^k}^2 \nonumber\\
&  \le \norm{\lambda^k  - \lambda}^2  + \frac{1}{L_{\text{d}}}
\langle \bar{\nabla} d_{\rho}(\lambda^k), \lambda^{k+1} - \lambda \rangle - \norm{\lambda^{k+1} - \lambda^k}^2 \nonumber\\
&  = \norm{\lambda^k  - \lambda}^2  + \frac{1}{L_{\text{d}}} \langle \bar{\nabla} d_{\rho}(\lambda^k) , \lambda^k - \lambda \rangle \label{rel_seq1}\\
& \quad \quad + \frac{1}{L_{\text{d}}} \left(\langle \bar{\nabla}
d_{\rho}(\lambda^k), \lambda^{k+1} - \lambda^{k} \rangle -
L_{\text{d}} \norm{\lambda^{k+1} - \lambda^k}^2\right)
\nonumber\\
& \le \norm{\lambda^k  -\lambda}^2  +
\frac{1}{L_{\text{d}}}(d_{\rho}(\lambda^{k+1}) - d_{\rho}(\lambda))
+ \frac{\ein}{L_{\text{d}}}.\nonumber
\end{align}
Taking now $\lambda = \lambda^*$ and  using an inductive argument,
we get:
\begin{equation}\label{bound_seq_dgm}
 \norm{\lambda^{k} - \lambda^*} \le R_d + \sqrt{\frac{k\ein}{L_{\text{d}}}}.
\end{equation}
On the other hand, for the scheme \textbf{DFGM}, we introduce the
notation $l^k = \lambda^{k-1} + \theta_{k}(\lambda^k -
\lambda^{k-1})$ and present an auxiliary result:

\begin{lemma}\cite{NecPat:15,Tse:08}
\label{th_tseng_2} Let $(\lambda^k, \mu^k)$ be generated by
Algorithm \textbf{DFOM}($d_{\rho},\calK_d$) with $\theta_{k+1} =
\frac{1 + \sqrt{1 + 4 \theta_k^2 }}{2}$, then for any Lagrange
multiplier $\lambda \in \rset^p$ we have:
\begin{align*}
\theta_{k}^2 (d_{\rho}(\lambda) - d_{\rho}(\lambda^{k})) \!+\!
\sum\limits_{j=0}^{k}\theta_{j}\Delta(\lambda,\mu^j) \!+\!
L_\text{d} \norm{l^{k} - \lambda}^2  \!\le\!  L_\text{d}
\norm{\lambda^0 - \lambda}^2  + 2 \! \sum\limits_{j=0}^{k}\theta_j^2
\ein,
\end{align*}
for all  $k \ge 0$,  where $\Delta(\lambda,\mu) =
\bar{d}_{\rho}(\mu) + \langle \bar{\nabla} d_{\rho}(\mu), \lambda -
\mu\rangle - d_{\rho}(\lambda)$.
\end{lemma}

\noindent Using this result and a similar reasoning as in
\cite{NecPat:15} we obtain the same relation \eqref{bound_seq_dgm}
for the scheme \textbf{DFGM}. Moreover, for simplicity, in the
sequel we also assume $\lambda^0=0$. In the next two sections we
derive  rate of convergence results of Algorithm \textbf{DFOM} in
both primal sequences, the primal last iterate \eqref{last_primal}
and an average of primal iterates \eqref{average_primal}.

%%%%%%%%%%%%%%%%%%%%%%%%%%%%%%%%%%%%%%%%%%%%%%%%%%%%%%%%%%%%%%%%%%%%%%%%%%%%%%%%%

\subsection{ The $\sqrt{\epsilon}$ convergence in the last primal iterate $\bar{u}_{\epsilon}^k$}
In this section we present rate of convergence  results for the
Algorithm \textbf{DFOM}, in terms of primal suboptimality and
infeasibility for the last primal iterate $\bar{u}_{\epsilon}^k$
defined in \eqref{last_primal}, provided that the relations
\eqref{choose_ein} hold.

\begin{theorem}
Let $\epsilon> 0$ be some desired accuracy and $\bar{u}_{\epsilon}^k =
\bar{u}(\lambda^k)$ be the primal last iterate sequence generated by Algorithm
\textbf{DFOM}($d_{\rho}, \mathcal{K}_d$)  using the inner accuracy from \eqref{choose_ein}.
Then, after  $k_{\text{out}}$   number of outer iterations given in \eqref{choose_ein},
$\bar{u}_{\epsilon}^{k_{\text{out}}}$ is $\sqrt{\epsilon}$-primal optimal for the original QP  \eqref{original_primal}.
\end{theorem}

\begin{proof}
Using the Lipschitz property of the gradient of $d_{\rho}(\cdot)$, it is known that the following inequality holds  \cite{Nes:04}:
\begin{equation*}
d_{\rho}(\lambda) \le d_{\rho}(\mu) + \langle \nabla d_{\rho}(\mu), \lambda - \mu \rangle
 - \frac{1}{2L_d} \norm{\nabla d_{\rho}(\lambda) - \nabla d_{\rho} (\mu)}^2 \quad \forall \lambda, \mu \in \rset^p.
\end{equation*}
Taking $\mu = \lambda^*$ and using the optimality condition $\langle \nabla d_{\rho}(\lambda^*), \mu - \lambda^*\rangle \le 0$
for all $\mu \in \mathcal{K}_d$, we further have:
\begin{equation}\label{grad_funcbound}
%\norm{G u(\lambda) - G u^*} =
\norm{\nabla d_{\rho}(\lambda) - \nabla d_{\rho}(\lambda^*)}
 \le \sqrt{2L_d (F^* - d_{\rho}(\lambda))} \qquad \forall \lambda \in \mathcal{K}_d.
\end{equation}
Considering $\lambda = \lambda^k$ and observing that
$ s(\bar{u}(\lambda^k),\lambda^k) + Gu^* + g \in \mathcal{K}$
we obtain a link between the primal feasibility
and dual suboptimality:
\begin{align}\label{aux_feas_bound}
 d_{\mathcal{K}}(G \bar{u}^k_{\epsilon}+g) & \le \left\| G \bar{u}(\lambda^k) + g  - s(\bar{u}(\lambda^k),\lambda^k) - G u^* - g \right\| \nonumber\\
 & = \norm{\bar{\nabla} d_{\rho}(\lambda^k)  - \nabla d_{\rho}(\lambda^*)}  \nonumber \\
&\le \norm{\bar{\nabla} d_{\rho}(\lambda^k) - \nabla d_{\rho} (\lambda^k)} + \norm{\nabla d_{\rho}(\lambda^k) - \nabla d_{\rho}(\lambda^*)} \nonumber\\
&  \overset{\eqref{grad_funcbound} + \eqref{inexact_gradient_rel}}{\le} \sqrt{2 L_d \ein} + \sqrt{2L_d (F^* - d_{\rho}(\lambda^k))}.
\end{align}

\noindent Provided that $F^* - d_{\rho}(\lambda^{k_{\text{out}}})
\le \epsilon$ and using $\ein$ as in \eqref{choose_ein}, we obtain:
\begin{equation}\label{pfeas_subopt_final}
d_{\mathcal{K}}(G \bar{u}^{k_{\text{out}}}_{\epsilon} + g) \le
\max \left\{ \frac{(L_d \epsilon)^{1/2}}{\sqrt{2}}, \frac{L_d^{1/4}}{(3R_d)^{1/2}} \epsilon^{3/4} \right\}  + (2L_d \epsilon)^{1/2}.
\end{equation}
\noindent Secondly, we find a link between the primal and dual
suboptimality. Indeed, we
have for all $\mathbf{\lambda} \in \mathcal{K}_d$:
\begin{align}\label{aux_subopt_left}
F^* & = \min_{u \in U, s \in \mathcal{K}} F(u)
+ \langle \mathbf{\lambda}^*, Gu + g - s\rangle  \nonumber \\
&\leq F(\bar{u}(\lambda^k)) +
\left \langle \mathbf{\lambda}^*, G \bar{u}(\lambda^k) + g - \left[G \bar{u}(\lambda^k) + g\right]_{\calK} \right \rangle.
\end{align}
Further, using the Cauchy-Schwartz inequality, we derive:
\begin{align} \label{psubopt_left}
F(\bar{u}^{k_{\text{out}}}_{\epsilon}) - F^*  & \geq - \|\lambda^*\| \text{dist}_{\mathcal{K}}(G \bar{u}(\lambda^{k_{\text{out}}}) + g) \nonumber \\
& \ge - R_d\max \left\{ \frac{(L_d \epsilon)^{1/2}}{\sqrt{2}}, \frac{L_d^{1/4}}{(3R_d)^{1/2}} \epsilon^{3/4} \right\}  - R_d(2L_d \epsilon)^{1/2}.
\end{align}
\noindent On the other hand, from the concavity  of
$d_{\rho}(\cdot)$ we obtain:
\begin{align}\label{aux_subopt_right}
&F(\bar{u}(\lambda^k)) - F^*
\le \bar{d}_{\rho}(\lambda^k) - F^* - \langle \bar{\nabla} d_{\rho}(\lambda^k), \lambda^k \rangle \nonumber\\
&\le d_{\rho}(\lambda^k) -F^* - \langle \nabla d_{\rho}(\lambda^*),
\lambda^k \rangle
 + \langle \nabla d_{\rho}(\lambda^*) - \bar{\nabla} d_{\rho}(\lambda^k), \lambda^k \rangle \nonumber + \ein \nonumber \\
&\le d_{\rho}(\lambda^k) -F^* - \langle \nabla d_{\rho}(\lambda^*),
\lambda^k -\lambda^*\rangle
 + \norm{\nabla d_{\rho}(\lambda^*) - \bar{\nabla} d_{\rho}(\lambda^k)}\norm{ \lambda^k} + \ein \nonumber \\
& \le \norm{\lambda^k}  \norm{\bar{\nabla} d_{\rho}(\lambda^k) - \nabla d_{\rho}(\lambda^*)} + \ein \nonumber \\
& \overset{\eqref{aux_feas_bound}}{\le} \norm{\lambda^k}\sqrt{2 L_d
\ein} + \norm{\lambda^k}\sqrt{2L_d (F^* - d_{\rho}(\lambda^k))}  + \ein.
\end{align}
Taking $k = k_{\text{out}}$ and $\ein$ as in \eqref{choose_ein},
based on \eqref{bound_seq_dgm} and on the implicit assumption that
$k_{\text{out}} \ge 1$, we observe that
$\norm{\lambda^{k_{\text{out}}}} \le \norm{\lambda^{k_{\text{out}}}
- \lambda^*} + \norm{\lambda^*} \le  4 R_d$ for both schemes
\textbf{DGM} and \textbf{DFGM}. Therefore, \eqref{aux_subopt_right}
implies:
\begin{equation*}
F(\bar{u}^{k_{\text{out}}}_{\epsilon}) - F^*
\overset{\eqref{pfeas_subopt_final}}{\le} 4R_d \max \left\{
\frac{(L_d \epsilon)^{1/2}}{\sqrt{2}},
\frac{L_d^{1/4}}{(3R_d)^{1/2}} \epsilon^{3/4} \right\}  + 4R_d(2L_d
\epsilon)^{1/2} + \ein.
\end{equation*}
\noindent  As a conclusion, from \eqref{psubopt_left} and the
previous inequality,  we get the bound:
\begin{align*}
| F(\bar{u}_{\epsilon}^{k_{\text{out}}}) - F^* | \le 4R_d \max
\left\{ \frac{(L_d \epsilon)^{1/2}}{\sqrt{2}},
\frac{L_d^{1/4}}{(3R_d)^{1/2}} \epsilon^{3/4} \right\}  + 4R_d(2L_d
\epsilon)^{1/2} + \ein,
 \end{align*}
which implies $| F (\bar{u}^{k_{\text{out}}}_{\epsilon}) - F^*| \le
\mathcal{O}(\sqrt{\epsilon})$. Using this fact and the feasibility
bound \eqref{pfeas_subopt_final}, which also implies
$\text{dist}_{\mathcal{K}} (G \bar{u}^{k_{\text{out}}}_{\epsilon} +
g) \le \mathcal{O}(\sqrt{\epsilon})$, we finally conclude that the
last primal iterate $\bar{u}^{k_{\text{out}}}_{\epsilon}$ is
$\sqrt{\epsilon}$-primal optimal. \qed
\end{proof}

\noindent We can also prove linear convergence for algorithm  \textbf{DFOM} provided that $\lambda_{\min} (Q) > 0$ (i.e. the objective function is smooth and strongly convex) and $U  = \rset^n$ (i.e. the inner problem is unconstrained). In this case we can show that the dual problem satisfies an error bound property \cite{NecNed:15,NecPat:15}. Under these settings \textbf{DFOM} is converging linearly (see \cite{NecNed:15,NecPat:15,WanLin:13} for more details).

%%%%%%%%%%%%%%%%%%%%%%%%%%%%%%%%%%%%%%%%%%%%%%%%%%%%%%%%%%%%%%%%%%%%%%%%%

\subsection{The $\epsilon$ convergence in the average of primal iterates $\hat{u}_{\epsilon}^k$}
Further, we analyze the convergence of the algorithmic framework
\textbf{DFOM} in the average of primal iterates
$\hat{u}^k_{\epsilon}$ defined in \eqref{average_primal}. Since we
consider different primal average iterates for the schemes
\textbf{DGM} and \textbf{DFGM}, we analyze separately the
convergence of these methods in  $\hat{u}^k_{\epsilon}$.

\begin{theorem}
Let $\epsilon> 0$ be some desired accuracy and
$\hat{u}_{\epsilon}^k$ be the primal average iterate given in
\eqref{average_primal}, generated by algorithm  \textbf{DGM}, i.e.
Algorithm \textbf{DFOM}($d_{\rho}, \mathcal{K}_d$) with $\theta_k =
1$ for all $k \ge 0$, using the inner accuracy from
\eqref{choose_ein}.  Then, after  $k_{\text{out}}$  number of outer
iterations given in \eqref{choose_ein},
$\hat{u}^{k_{\text{out}}}_{\epsilon}$ is $\epsilon$-primal  optimal
for the original QP  \eqref{original_primal}.
\end{theorem}

\begin{proof}
First, we derive sublinear estimates for   primal infeasibility for
the average primal sequence $\hat{u}^k_{\epsilon}$ (recall that in
this case  $\hat{u}^k_{\epsilon} =
\frac{1}{k+1}\sum\limits_{j=0}^{k} \bar{u}^j$). Given the definition
of $\lambda^{j+1}$ in Algorithm \textbf{DFOM}($d_{\rho}, \calK_d$)
with $\theta_j = 1$, we get:
\[   \lambda^{j+1} = \left[ \lambda^j + \frac{1}{2L_{\text{d}}} \bar{\nabla} d_{\rho}(\lambda^j) \right]_{\mathcal{K}_d}  \quad \forall j \geq 0.\]
Subtracting $\lambda^j$ from both sides, adding up the above
inequality for $j=0$ to $j=k$, we obtain:
\begin{align}\label{bound_feas}
\left\|\frac{1}{k+1}\sum_{j=0}^k \left[\lambda^j + \frac{1}{2L_d}
\bar{\nabla} d_{\rho}(\lambda^j) \right]_{\mathcal{K}_d} -\lambda^j
\right\| = \frac{1}{k+1}\norm{\lambda^{k+1} - \lambda^0} .
\end{align}
If we denote $z^j =  \lambda^j + \frac{1}{2L_d}\bar{\nabla}
d_{\rho}(\lambda^j)  - \left[ \lambda^j + \frac{1}{2L_d}\bar{\nabla}
d_{\rho}(\lambda^j) \right]_{\mathcal{K}_d} $, then we observe that
$z^j \in \mathcal{K}$. Thus, we have
$\frac{2L_d}{k+1}\sum\limits_{j=0}^{k} z^j \in  \mathcal{K}$. Using
the definition of  $ \bar{\nabla} d_{\rho}(\lambda^j)$, we obtain:
\begin{align*}%\label{feasibility_prelim}
 \text{dist}_{\mathcal{K}}(G\hat{u}^k_{\epsilon} + g) &\le
 \left\|  \frac{1}{k+1}\sum_{j=0}^k (G \bar{u}^j + g)
 - \frac{1}{k+1}\sum\limits_{j=0}^k \left(2L_d z^j +
  s(\bar{u}^j,\lambda^j) \right)  \right\| \nonumber  \\
  & = \left\| \frac{1}{k+1} \sum \limits_{j=0}^k (\bar{\nabla} d_{\rho}(\lambda^j)
   - 2L_d z^j) \right\|
 \overset{\eqref{bound_feas}}{=} \frac{2 L_{\text{d}}}{k+1}\norm{\lambda^{k+1} - \lambda^0}.
\end{align*}
Using $\norm{\lambda^k - \lambda^0} \le \norm{\lambda^k - \lambda^*}
+ R_d$  and the bound \eqref{bound_seq_dgm} for the values $\ein$
and $k = k_{\text{out}}$ from \eqref{choose_ein} in the previous
inequality, we get:
\begin{equation}\label{feasibility_final}
 \text{dist}_{\mathcal{K}}(G\hat{u}^{k_{\text{out}}}_{\epsilon}+g) \le
\frac{4L_{\text{d}} R_{\text{d}}}{k_{\text{out}}} +
 2\sqrt{\frac{L_{\text{d}}\ein}{k_{\text{out}}}} \le \frac{\epsilon}{R_d}.
\end{equation}

\noindent It remains to estimate the primal suboptimality. First, to
bound below $F(\hat{u}^{k_{\text{out}}}_{\epsilon}) - F^*$ we
proceed as follows:
\begin{align*}
 F^* &= \min\limits_{u \in U, s \in \mathcal{K}} F(u) + \langle \lambda^*,
 Gu+g -s \rangle \nonumber\\
& \le F(\hat{u}^k_{\epsilon}) + \langle \lambda^*, G\hat{u}^k_{\epsilon} +
 g - \left[G\hat{u}^k_{\epsilon} + g\right]_{\mathcal{K}} \rangle \nonumber\\
& \le F(\hat{u}^k_{\epsilon}) + \norm{\lambda^*} \norm{G\hat{u}^k_{\epsilon}+g -
 \left[G\hat{u}^k_{\epsilon} + g\right]_{\mathcal{K}}}\nonumber\\
& = F(\hat{u}^k_{\epsilon}) + R_d \;  \text{dist}_{\mathcal{K}}
\left(G\hat{u}^k_{\epsilon} + g\right).
 \end{align*}
Combining the last inequality  with \eqref{feasibility_final}, we
obtain:
\begin{equation}\label{left_subopt}
 - \epsilon \le  F(\hat{u}^{k_{\text{out}}}_{\epsilon}) - F^*.
\end{equation}
Secondly, we observe the following facts: for any $u\in U$,
$d_{\rho}(\lambda) \le F^*$ and the following identity holds:
\begin{align}\label{identity_lag}
 \bar{d}_{\rho}(\lambda) - \langle \bar{\nabla} d_{\rho}(\lambda), \lambda \rangle = F(\bar{u}(\lambda)) + \frac{\rho}{2}\norm{\bar{\nabla} d_{\rho}(\lambda)}^2 \ge F(\bar{u}(\lambda)).
\end{align}
\noindent Based  on previous discussion,  \eqref{rel_seq1} and
\eqref{identity_lag}, we derive that
\begin{align*}
& \norm{\lambda^{k+1} - \lambda}^2 \\
& \overset{\eqref{rel_seq1}}{\le} \norm{\lambda^k - \lambda}^2 + \frac{1}{L_{\text{d}}}
 \left( d_{\rho}(\lambda^{k+1}) - \bar{d}(\lambda^k) + \langle \bar{\nabla} d_{\rho}(\lambda^k), \lambda^k - \lambda \rangle + \ein  \right)\\
 & \overset{\eqref{identity_lag}}{\le} \norm{\lambda^k - \lambda}^2 + \frac{1}{L_{\text{d}}}\left( F^* - F(\bar{u}^k) - \frac{\rho}{2}\norm{\bar{\nabla} d_{\rho}(\lambda^k)}^2 + \ein - \langle
\bar{\nabla} d_{\rho}(\lambda^k), \lambda \rangle   \right).
 \end{align*}
Taking now $\lambda=0$, $k = k_{\text{out}}$ and using an  inductive
argument, we obtain:
\begin{equation}\label{right_subopt}
 F(\hat{u}^{k_{\text{out}}}) - F^* \le
 \frac{L_{\text{d}}\norm{\lambda^0}^2}{k_{\text{out}} } +
 \ein =  \frac{\epsilon}{4},
\end{equation}
provided that $\lambda^0=0$.  From \eqref{feasibility_final},
\eqref{left_subopt} and \eqref{right_subopt}, we obtain that the
average primal iterate  $\hat{u}^{k_{\text{out}}}_{\epsilon}$ is
$\epsilon$-primal optimal.  \qed
\end{proof}

\noindent Further, we analyze the primal convergence rate  of
algorithm \textbf{DFGM} in the average primal iterate
$\hat{u}_{\epsilon}^k$:

\begin{theorem}
Let $\epsilon> 0$ be some desired accuracy and
$\hat{u}_{\epsilon}^k$ be the primal average iterate given in
\eqref{average_primal}, generated  by algorithm \textbf{DFGM}, i.e.
Algorithm \textbf{DFOM}($d_{\rho}, \mathcal{K}_d$) with
$\theta_{k+1} = \frac{1 + \sqrt{1 + 4 \theta_k^2}}{2}$ for all $k
\geq0$, using the inner accuracy from \eqref{choose_ein}. Then,
after $k_{\text{out}}$
 number of outer iterations given in \eqref{choose_ein},
$\hat{u}^{k_{\text{out}}}_{\epsilon}$ is $\epsilon$-primal optimal
for the original QP \eqref{original_primal}.
\end{theorem}
\begin{proof}
Recall that we have defined  $S_k = \sum\limits_{j=0}^k \theta_k$.
Then, it follows:
\begin{equation}\label{fg_step}
 \frac{k+1}{2} \le \theta_k \le k \qquad \text{and} \qquad S_k = \theta_k^2.
\end{equation}
\noindent For any $j \ge 0$ we denote $z^j = \mu^{j} +
\frac{1}{2L_{\text{d}}}  \bar{\nabla} d_{\rho}(\mu^{j})$ and thus we
have $\lambda^j = \left[ z^j \right]_{\mathcal{K}_d}$. In these
settings, we have the following relations:
\begin{align} \label{feasibility_aux1}
\theta_j & \left( \frac{1}{2L_{\text{d}}} \bar{\nabla} d_{\rho}(\mu^j)- (z^j - [z^j]_{\mathcal{K}_d})\right) \nonumber\\
& = \theta_j \left(\left[ \mu^{j} + \frac{1}{2L_{\text{d}}} \bar{\nabla} d_{\rho}(\mu^{j}) \right]_{\mathcal{K}_d} - \lambda^{j}\right) \nonumber\\
& =  \theta_{j}(\lambda^{j} - \mu^{j})\nonumber\\
& = \theta_{j}(\lambda^{j} - \lambda^{j-1}) + (\theta_{j-1} -1)( \lambda^{j-2} - \lambda^{j-1})\nonumber\\
& = \underbrace{\lambda^{j-1} + \theta_{j}(\lambda^{j} -
\lambda^{j-1})}_{=l^{j}} - \underbrace{(\lambda^{j-2} +
\theta_{j-1}(\lambda^{j-1} - \lambda^{j-2}))}_{=l^{j-1}}.
\end{align}
\noindent For simplicity consider $\lambda^{-2} = \lambda^{-1} = \lambda^0$ and $\theta_{-1} = \theta_0 = 0$.
Adding up the above equality for $j = 0$ to $j = k$,
multiplying by $\frac{2L_{\text{d}}}{S_k}$ and
observing that $ s(\bar{u}^j,\mu^j) + z^j - [z^j]_{\mathcal{K}_d} \in \mathcal{K}$ for all $j \ge 0 $, we obtain:
\begin{align*}
\text{dist}_{\mathcal{K}}\left(G\hat{u}^k_{\epsilon} + g \right)
&\le \left\| \sum\limits_{j=0}^k  \frac{\theta_j}{S_k}  \left( G \bar{u}^j + g -
 s(\bar{u}^j,\mu^j) - 2 L_{\text{d}}(z^j-[z^j]_{\mathcal{K}_d}) \right)\right\| \\
& =  \left\| \sum\limits_{j=0}^k  \frac{\theta_j}{S_k} \left( \bar{\nabla} d_{\rho}(\mu^j) -
2 L_{\text{d}}(z^j-[z^j]_{\mathcal{K}_d}) \right)\right\| \\
& \overset{\eqref{feasibility_aux1}}{=} \frac{L_\text{d}}{S_k}\norm{l^{k}-l^0} \le \frac{4L_\text{d}}{(k+1)^2} \norm{l^k -l^{0}}.
\end{align*}
\noindent Taking $\lambda =\lambda^*$ in Lemma \ref{th_tseng_2} and
using that the two  terms $\theta_{k}^2 (F^* -
d_{\rho}(\lambda^{k}))$ and $\sum\limits_{j=0}^{k}\theta_{j}
 \Delta(\lambda^*,\mu^j)$ are positive, we get:
\begin{align*}
\| l^k - \lambda^* \| &\le \sqrt{\| \lambda^0 - \lambda^*\|^2 +
\sum\limits_{i=1}^{k} \frac{2\theta_i^2\ein}{L_{\text{d}}}} \le \| \lambda^0 - \lambda^*\| + \sqrt{\frac{8\ein}{3L_{\text{d}}} (k+1)^3}\\
& \le \| \lambda^0 - \lambda^* \| +
\left(\frac{8\ein}{3L_{\text{d}}}\right)^{1/2} (k+1)^{3/2} \qquad
\forall k \geq 0.
\end{align*}

\noindent Thus, we can further bound the primal infeasibility as
follows:
\begin{align}
\text{dist}_{\mathcal{K}}\left(G\hat{u}^k_{\epsilon} + g \right)
&\le \frac{4L_\text{d}}{(k+1)^2} \|l^k - l^0\|
\le \frac{4L_\text{d}}{(k+1)^2}(\|l^k - \lambda^*\| + R_d) \nonumber\\
& \le \frac{8L_\text{d}R_{\text{d}}}{(k+1)^2}  +
 8 \left( \frac{L_d \ein }{k+1} \right)^{1/2}.
\label{infes_av}
\end{align}
Therefore, using $k_{\text{out}}$ and $\ein$ from
\eqref{choose_ein}, it can be derived  that:
\begin{equation}\label{infes_av_final}
 \text{dist}_{\mathcal{K}}(G \hat{u}^{k_{\text{out}}}_{\epsilon} + g) \le \frac{8L_{\text{d}} R_{\text{d}}}{k_{\text{out}}^2} + 8 \left(\frac{L_d \ein}{k_{\text{out}}} \right)^{1/2} \le \frac{3 \epsilon}{R_d}.
\end{equation}
\noindent Further, we derive sublinear  estimates for primal
suboptimality. First, note the following relations:
\begin{align*}
\Delta(\lambda, & \mu^{k}) =  \bar{d}_{\rho}(\mu^{k}) + \langle \bar{\nabla} d_{\rho}(\mu^{k}),
\lambda - \mu^{k}\rangle - d_{\rho}(\lambda) \\
&= \mathcal{L}_{\rho}(\bar{u}^{k},\mu^{k}) + \langle \bar{\nabla}d_{\rho}(\mu^{k}) , \lambda - \mu^{k}\rangle
- d_{\rho} (\lambda) \\
& = F(\bar{u}^k) + \langle \lambda, G\bar{u}^k + g - s(\bar{u}^k,\mu^k) \rangle + \frac{\rho}{2}\norm{G \bar{u}^k +g - s(\bar{u}^k,\mu^k)}^2 - d_{\rho}(\lambda)\\
& \ge \min\limits_{s \in \calK} \; \; F(\bar{u}^k) + \langle \lambda, G\bar{u}^k + g - s\rangle + \frac{\rho}{2}\norm{G \bar{u}^k +g - s}^2 - d_{\rho}(\lambda)\\
& = \mathcal{L}_{\rho}(\bar{u}^{k},\lambda) - d_{\rho} (\lambda).
 \end{align*}
\noindent Summing on the history  and using  the convexity of
$\mathcal{L}_{\rho}(\cdot,\lambda)$, we get:
\begin{align}
\sum\limits_{j=0}^{k}&\theta_j \Delta(\lambda,\mu^j)
\ge \sum\limits_{j=1}^{k}\theta_j ( \mathcal{L}_{\rho}(\bar{u}^{j},\lambda) - d_{\rho}(\lambda))\nonumber\\
&\ge S_{k} \left( \mathcal{L}_{\rho} (\hat{u}^{k}_{\epsilon},\lambda) -d_{\rho} (\lambda)\right)
= \theta_{k}^2 \left( \mathcal{L}_{\rho} (\hat{u}^{k}_{\epsilon} , \lambda) - d_{\rho} (\lambda)\right).
\label{sum_theta_aux_ag}
\end{align}
Using \eqref{sum_theta_aux_ag} in Lemma \ref{th_tseng_2}, and
dropping the term $ L_{\text{d}}\norm{l^{k}- \lambda}^2$, we have:
\begin{align}\label{subopt_right_aux}
F(\hat{u}^{k}_{\epsilon}) + \langle G\hat{u}^{k}_{\epsilon} +g - s(\hat{u}^k_{\epsilon},\lambda),& \lambda \rangle - d_{\rho}(\lambda^{k})
\le \frac{L_\text{d}}{\theta_{k}^2}\norm{\lambda^0- \lambda}^2 + \frac{2\sum\limits_{j=0}^{k} \theta_j^2}{\theta_{k}^2} \ein.
\end{align}
Moreover, we have that:
\[\frac{1}{\theta_{k}^2}\sum\limits_{j=0}^{k} \theta_j^2 =
\frac{1}{S_{k}} \sum\limits_{j=0}^{k} \theta_j^2 \le
\max\limits_{0\le j \le k} \theta_j \le k \quad \text{and} \quad
d_{\rho} (\lambda^{k}) \le F^*. \]  Now, by choosing the Lagrange
multiplier $\lambda = 0$ and $k = k_{\text{out}}$ in
\eqref{subopt_right_aux}, we have:
\begin{align}\label{subopt_right_av}
F(&\hat{u}^{k_{\text{out}}}_{\epsilon}) - F^* \le
F(\hat{u}^{k_{\text{out}}}_{\epsilon}) - d_{\rho}
(\lambda^{k_{\text{out}}}) \le
\frac{2L_\text{d}R_{\text{d}}^2}{k^2_{\text{out}}} +
2k_{\text{out}}\ein \le \frac{5\epsilon}{4}.
\end{align}

\noindent On the other hand, we have:
\begin{align*}
F^* &=  \min_{u \in U, s \in \mathcal{K}} F(u) + \langle \lambda^*, G u + g -s \rangle \leq F(\hat{u}^k_{\epsilon}) + \langle \lambda^*, G \hat{u}^k_{\epsilon} +g - \left[G \hat{u}^k_{\epsilon} + g \right]_{\mathcal{K}} \rangle\nonumber\\
& \le F(\hat{u}^k_{\epsilon}) + R_d \; \text{dist}_{\calK}(G
\hat{u}^k_{\epsilon} + g).
\end{align*}
Taking $k  = k_{\text{out}}$ and $\ein$ from \eqref{choose_ein}, and using \eqref{infes_av_final}, we obtain:
\begin{equation}\label{subopt_left_final}
   -3 \epsilon \le F(\hat{u}^{k_{\text{out}}}_{\epsilon}) - F^*.
\end{equation}
Finally, from \eqref{infes_av_final}, \eqref{subopt_right_av} and
\eqref{subopt_left_final}, we get  that the primal average sequence
$\hat{u}^{k_{\text{out}}}_{\epsilon}$ is $\epsilon$ primal optimal.
\qed
\end{proof}

\noindent In conclusion, in DuQuad we generate two approximate
primal solutions $\bar{u}_\epsilon^k$ and $\hat{u}_\epsilon^k$ for
each algorithm \textbf{DGM} and \textbf{DFGM}. From previous
discussion it can be seen that theoretically, the average of primal
iterates sequence $\hat{u}_\epsilon^k$  has a better behavior than
the last iterate sequence $\bar{u}_\epsilon^k$. On the other hand,
from our practical experience (see also Section
\ref{numerical_tests}) we have observed that usually dual first
order  methods are converging faster in the primal last iterate than
in a primal average sequence. Moreover, from our unified analysis we
can conclude that for both approaches, ordinary dual with  $Q \succ
0$ and augmented dual with $Q \succeq 0$, the rates of convergence
of algorithm \textbf{DFOM} are the same.

%%%%%%%%%%%%%%%%%%%%%%%%%%%%%%%%%%%%%%%%%%%%%%%%%%%%%%%%%%%%%%%%%%%%%%%%%%%%
%%%%%%%%%%%%%%%%%%%%%%%%%%%%%%%%%%%%%%%%%%%%%%%%%%%%%%%%%%%%%%%%%%%%%%%%%%%%%%

\section{Total computational complexity in  DuQuad}
\noindent In this section we derive the total computational
complexity of the algorithmic framework  \textbf{DFOM}. Without lose
of generality, we make the assumptions: $ R_d>1, \epsilon<1, \lambda_{\text{max}}(Q) \ge \norm{G}^2.$
However, if any of these assumptions does not hold, then our result
are still  valid with minor changes in constants. Now, we are ready
to derive the total number of iterations for  \textbf{DFOM}, i.e.
the total number of projections on the set $U$ and of matrix-vector
multiplications $Qu$ and $G^T \lambda$.

\begin{theorem}\label{in:th_last}
Let $\epsilon> 0$ be some desired accuracy and  the inner accuracy
$\ein$ and the number of outer iterations $k_{\text{out}}$ be as in
\eqref{choose_ein}.  By setting $\rho = \frac{8 R_d^2}{\epsilon}$
and assuming that the primal iterate $\bar{u}^k$ is obtained by
running the Algorithm \textbf{FOM}($\mathcal{L}_{\rho}(\cdot,
\mu^k), U$), then $\bar{u}_{\epsilon}^k$ ($\hat{u}_{\epsilon}^k$ )
is $\sqrt{\epsilon}$ ($\epsilon$) primal optimal after a total
number of projections on the set $U$ and of matrix-vector
multiplications $Qu$ and $G^T \lambda$ given by:
\begin{equation*}
 k_{\text{total}} = \begin{cases}
\left\lfloor \frac{24 \norm{G} D_U R_d}{\epsilon} \right\rfloor
&\text{if} \quad \sigma_{\mathcal{L}} = 0 \\
\left\lfloor \frac{16 \norm{G}
R_d}{\sqrt{\sigma_{\mathcal{L}}\epsilon }}  \log \left( \frac{8
\norm{G} D_U R_d}{\epsilon} \right) \right\rfloor &\text{if} \quad
\sigma_{\mathcal{L}} > 0.
\end{cases}
\end{equation*}
\end{theorem}

\begin{proof}
From Lemma \ref{lemma_sublin_dfg} we have that the inner problem
(i.e.  finding the primal iterate $\bar{u}^k$) for a given $\mu^k$
can be solved in sublinear (linear) time using Algorithm {\bf
FOM}($\lu_\rho(\mu^k,\cdot),U$),  provided that  the inner problem
has smooth (strongly) convex objective function, i.e.
$\lu_\rho(\mu^k,\cdot)$ has  $\sigma_{\lu} =0$ ($\sigma_{\lu} >0$).
More precisely, from  Lemma \ref{lemma_sublin_dfg}, it follows that,
regardless if we apply algorithms \textbf{DFGM} or~\textbf{DGM},  we
need to perform the following number of inner iterations  for
finding the primal iterate $\bar{u}^k$ for a given $\mu^k$:
\[ k_{\text{in}} = \begin{cases}
\sqrt{\frac{2 L_{\mathcal{L}} D_U^2}{\ein}}, & \text{if}\;\; \sigma_{\mathcal{L}} =0 \\
\sqrt{\frac{L_{\mathcal{L}} }{\sigma_{\mathcal{L}}}}
\log\left(\frac{L_{\mathcal{L}} D_U^2}{\ein}\right) + 1, &
\text{if}\;\;  \sigma_{\mathcal{L}}>0.
\end{cases} \]
\noindent Combining these estimates with the expressions
\eqref{choose_ein}   for the inner accuracy $\ein$, we obtain, in
the first case $\sigma_{\mathcal{L}} = 0$, the following inner
complexity estimates:
\begin{equation*}
 k_{\text{in}} =  \begin{cases}\left(\frac{8 L_{\mathcal{L}} D_U^2}{\epsilon}\right)^{1/2} & \text{if} \;\; \textbf{DGM} \\
\frac{4 (L_{\mathcal{L}} D_U^2)^{1/2} (2L_d R_d^2)^{1/4} }{\epsilon^{3/4}}, & \text{if} \;\; \textbf{DFGM}.
\end{cases}
\end{equation*}
Multiplying $k_{\text{in}}$ with the number of outer iterations
$k_{\text{out}}$ from \eqref{choose_ein} and minimizing the product
$k_{\text{in}} k_{\text{out}} $ over the smoothing parameter
$\rho$ (recall that $L_{\lu} = \lambda_{\max} (Q) + \rho \|G\|^2$ and $L_{\text{d}}= \frac{\|G\|^2}{\lambda_{\min}(Q) +
\rho\|G\|^2}$), we obtain the following optimal computational complexity
estimate (number of projections on the set $U$ and evaluations of
$Qu$ and $G^T \lambda$):
\begin{equation*}
 k_{\text{total}}^*  = (k_{\text{out}} k_{\text{in}})^* = \frac{24 \norm{G} D_U
 R_d}{\epsilon},
\end{equation*}
which is attained for the optimal parameter choice:
\[ \rho^* = \frac{8 R_d^2}{\epsilon}. \]

\noindent Using the same reasoning for the second  case when
$\sigma_{\mathcal{L}} > 0$, we observe that the value $\rho =
\frac{8 R_d^2}{\epsilon}$ is also optimal for this case in the
following sense: the difference between the estimates obtained with
the exact optimal $\rho$ and the value $\frac{8 R_d^2}{\epsilon}$
are only minor changes in constants. Therefore,  when
$\sigma_{\mathcal{L}} > 0$, the total  computational complexity
(number of projections on the set $U$ and evaluations of $Qu$ and
$G^T \lambda$) is:
\begin{equation*}
k_{\text{total}}^*  = (k_{\text{out}} k_{\text{in}})^* = \frac{16
\norm{G} R_d}{\sqrt{\sigma_{\mathcal{L}}\epsilon }} \log \left(
\frac{8 \norm{G} D_U R_d}{\epsilon} \right).
\end{equation*}
\qed
\end{proof}

\noindent In conclusion,  the last primal iterate
$\bar{u}_{\epsilon}^k$ is $\sqrt{\epsilon}$-primal optimal after
${\cal O} (\frac{1}{\epsilon})$  (${\cal
O}(\frac{1}{\sqrt{\epsilon}} \log \frac{1}{\epsilon})$) total number
of projections on the set $U$ and of matrix-vector multiplications
$Qu$ and $G^T \lambda$, provided that $\sigma_{\mathcal{L}} = 0$
($\sigma_{\mathcal{L}} > 0$).  Similarly,  the average of primal
iterate $\hat{u}_{\epsilon}^k$ is $\epsilon$-primal optimal after
${\cal O} (\frac{1}{\epsilon})$  (${\cal
O}(\frac{1}{\sqrt{\epsilon}} \log \frac{1}{\epsilon})$) total number
of projections on the set $U$ and of matrix-vector multiplications
$Qu$ and $G^T \lambda$, provided that $\sigma_{\mathcal{L}} = 0$
($\sigma_{\mathcal{L}} > 0$). Moreover, the optimal choice for the
parameter $\rho$ is of order ${\cal O}(\frac{1}{\epsilon})$, provided that
$\lambda_{\min}(Q) =0$.

%%%%%%%%%%%%%%%%%%%%%%%%%%%%%%%%%%%%%%%%%%%%%%%%%%%%%%%%%%%%%%

\subsection{What is the main computational bottleneck in DuQuad?}
\noindent Let us analyze now the  computational cost per inner and
outer iteration for Algorithm {\bf DFOM}($d_\rho,\calK^*$) for
solving approximately the original QP \eqref{original_primal}:

\vspace{0.2cm}

\noindent \textbf{Inner iteration}: When we solve the inner problem
with the Nesterov's algorithm {\bf FOM}($\lu_\rho(\mu,\cdot),U$),
the main computational effort is done in computing the gradient of
the augmented Lagrangian $\lu_\rho(\mu,\cdot)$ defined in
\eqref{auglag}, which e.g. has the form:
\[ \nabla \lu_\rho(\mu,u) = (Q + \rho G^T G)u +
(q + G^T \mu + \rho G^T g).  \] In DuQuad these
matrix-vector operations are implemented efficiently  in C (matrices
that do not change along iterations are computed once and only
$G^T \mu$ is computed at each outer iteration). The cost for
computing $\nabla \lu_\rho(\mu,u)$ for general QPs is ${\cal O}
(n^2)$. However, when the matrices $Q$  and $G$ are sparse
(e.g. network utility maximization problem) the cost ${\cal O}
(n^2)$ can be reduced substantially.  The other operations in
Algorithm {\bf FOM}($\lu_\rho(\mu,\cdot),U$) are just vector
operations and thus they are of order ${\cal O} (n)$. Thus, the
dominant operation at the inner stage is the matrix-vector product.

\vspace{0.2cm}

\noindent \textbf{Outer iteration}: When solving the outer (dual)
problem with Algorithm {\bf DFOM}($d_\rho,\calK^*$), the main
computational effort is done in computing the inexact gradient of
the dual function:
\[ \bar{\nabla} d_\rho(\mu) = G \bar{u}(\mu) + g - s(\bar{u}(\mu),\mu).  \]
The cost for computing $\bar{\nabla} d_\rho(\mu)$ for general QPs is
${\cal O} (np)$. However, when the matrix  $G$ is sparse, this cost
can be reduced.  The other operations in Algorithm {\bf
DFOM}($d_\rho,\calK_d$) are of order ${\cal O} (p)$. Thus the
dominant operation  at the outer stage is also the matrix-vector
product.

\noindent Fig. \ref{fig:gprof_n150_dfgm_case1} displays the result
of profiling the code with gprof. In this simulation, a standard QP
with inequality constraints and dimensions $n = 150$ and $p = 225$
was solved by Algorithm \textbf{DFGM}. The profiling summary is
listed in the order of the time spent in each file.  This figure
shows that almost all the time for executing the program is spent in
the  library module \textit{math-functions.c}. Furthermore,
\textit{mtx-vec-mul} is by far the dominating function in this list.
This function is multiplying a matrix with a vector, which is
defined as a special type of matrix multiplication.
\begin{figure}[h!]
\centering
\includegraphics[width=0.45\textwidth,height=6cm]{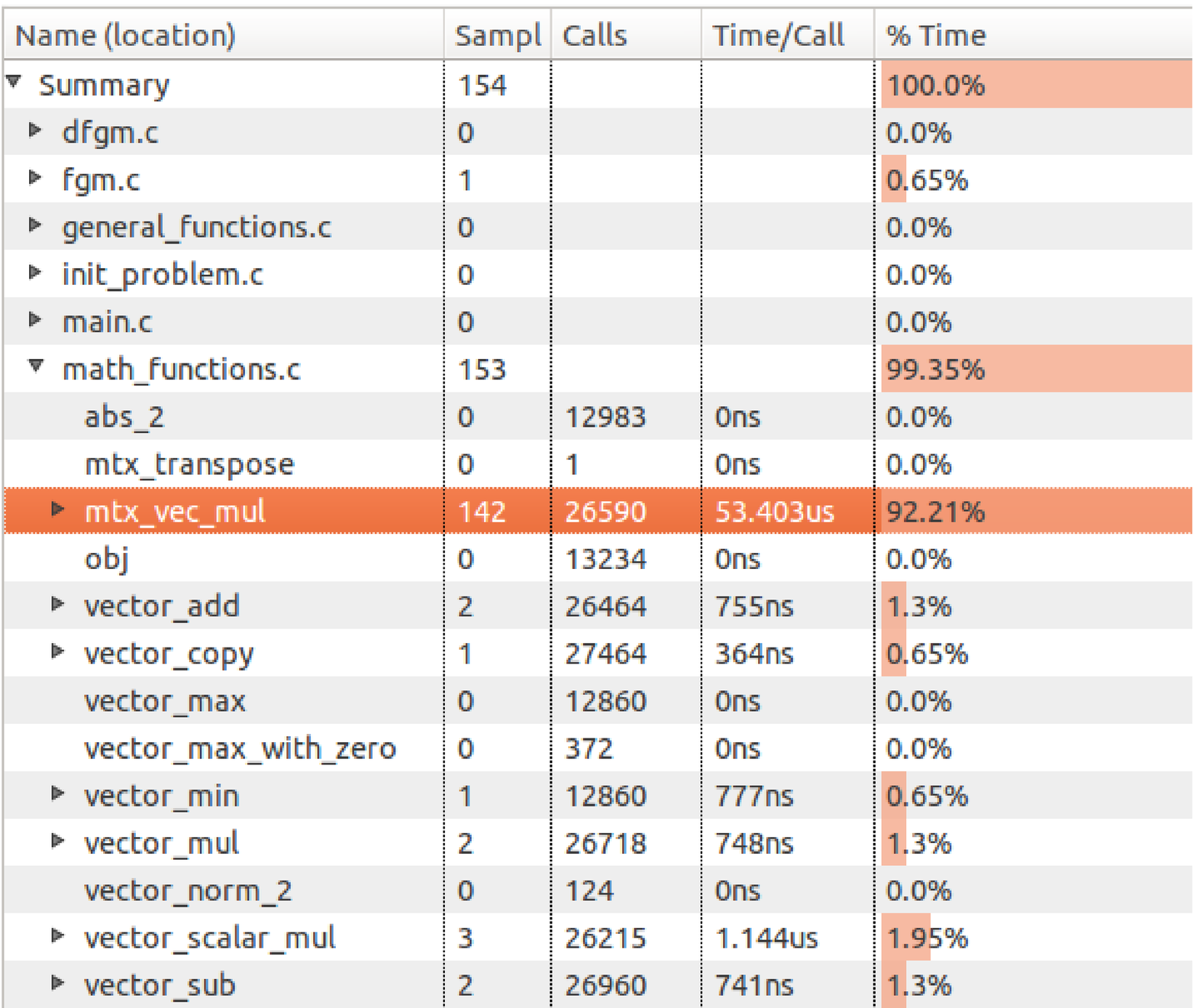}
\caption{Profiling the code with gprof.}
\label{fig:gprof_n150_dfgm_case1}
\end{figure}

\noindent In conclusion, in DuQuad the main operations are the
matrix-vector products. Therefore, DuQuad is adequate for solving QP
problems on hardware  with limited resources and capabilities, since
it does not require any solver for linear systems or other
complicating operations, while  most of the existing solvers for QPs
from the literature implementing e.g. active set or interior point
methods require the capability of solving linear systems. On the
other hand, DuQuad can be also used for solving large-scale sparse
QP problems since the iterations are very cheap in this case (only
sparse matrix-vector products).

%%%%%%%%%%%%%%%%%%%%%%%%%%%%%%%%%%%%%%%%%%%%%%%%%%%%%%
%%%%%%%%%%%%%%%%%%%%%%%%%%%%%%%%%%%%%%%%%%%%%%%%%%%%%%

\section{Numerical simulations}
\label{numerical_tests} DuQuad is mainly intended for  small to
medium size, dense QP problems, but it is of course also possible to
use DuQuad to solve (sparse) QP instances of large dimension.

\subsection{Distribution of DuQuad}

The DuQuad software package is available for download from:\\

\textcolor{blue}{http://acse.pub.ro/person/ion-necoara}\\

\noindent and distributed under general public license to allow
linking against proprietary codes. Proceed to the menu point
``Software'' to obtain a zipped archive of the most current version
of DuQuad. The users manual and extensive
source code documentation  are available here as well.\\

\noindent An overview of the workflow in DuQuad is illustrated in
Fig. \ref{fig:duquad_workflow}. A QP problem is constructed using a
Matlab script called \textit{test.m}. Then, the function
\textit{duquad.m} is called with the problem as input and is
regarded as a prepossessing stage for the online optimization. The
binary MEX file is called, with the original problem and the extra
info as input. The \textit{main.c} file of the C-code includes the
MEX framework and is able to convert the MATLAB data into C format.
Furthermore, the converted data gets bundled into a C struct and
passed as input to the algorithm that solves the~problem.
\begin{figure}[h!]
    \centering
    \includegraphics[width=0.5\textwidth]{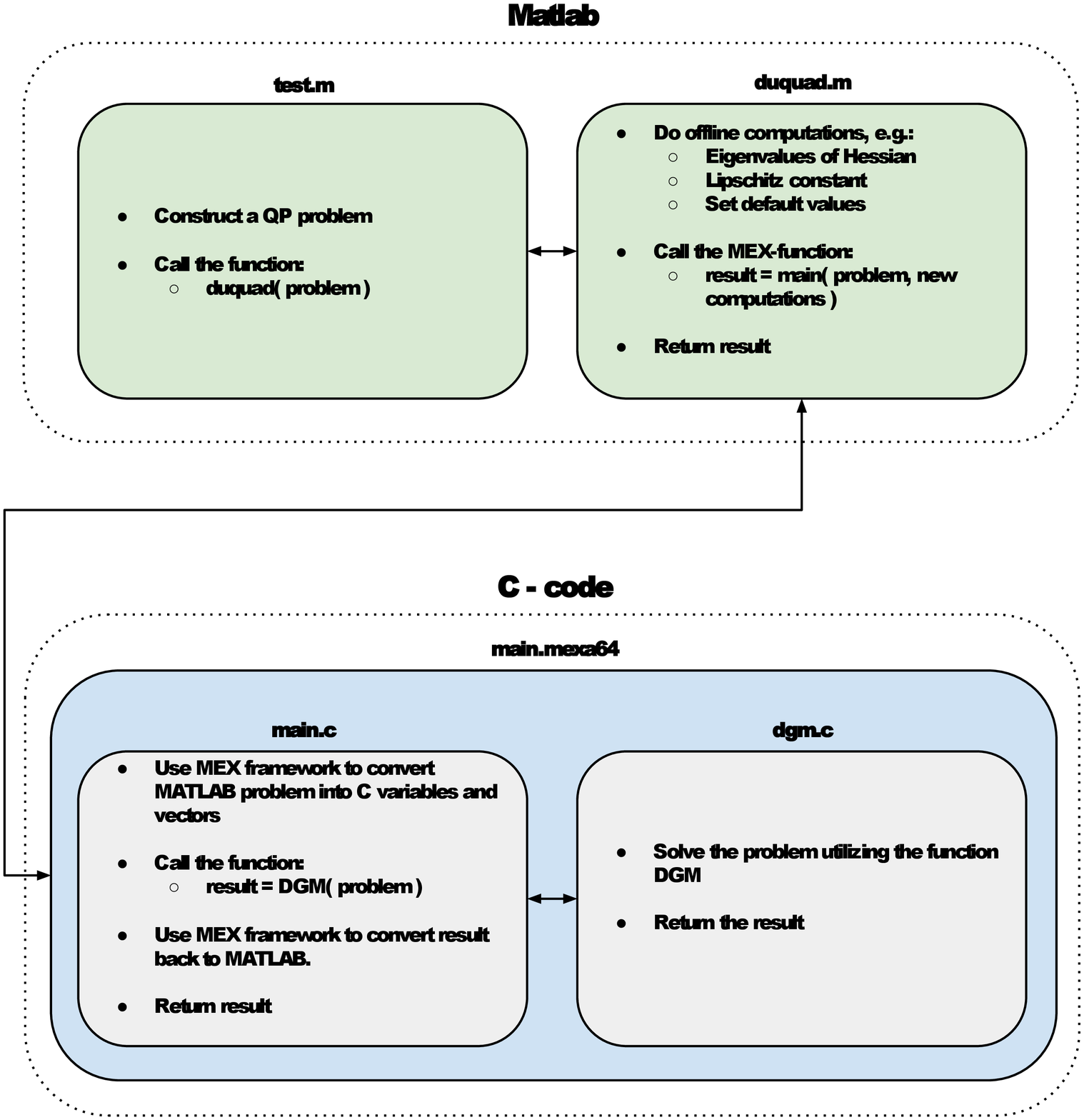}
    \caption{DuQuad workflow.}
    \label{fig:duquad_workflow}
\end{figure}

%%%%%%%%%%%%%%%%%%%%%%%%%%%%%%%%%%%%%%%%%%%%%%%%%%%%%%%%%%%%%%%%%%%%%%

% \subsection{Implementable DFOM in DuQuad}
% The   Algorithm {\bf DFOM} is dependent on $R_d = \|\lambda^*\|$,
% provided that $\lambda^0=0$. In conclusion, we need a good estimate
% on $\|\lambda^*\|$. In this section we propose an implementable
% Algorithm {\bf DFOM} based on a search procedure.......

%%%%%%%%%%%%%%%%%%%%%%%%%%%%%%%%%%%%%%%%%%%%%%%%%%%%%%%%%%%%%%%

\subsection{Numerical tests: case $\calK=\{0\}$}
\noindent We plot in Fig. \ref{fig:qp_cpu} the average CPU time for
several solvers, obtained by solving $50$ random QP's with equality
constraints ($Q \succeq 0$ and $\calK=\{0\}$) for each
dimension $n$, with an accuracy $\epsilon=0.01$ and the stopping
criteria $\abs{F(\hat{u}^k_\epsilon) - F^{*}}$ and $\norm{G
\hat{u}^k_\epsilon + g}$ less than the accuracy $\epsilon$. In
both algorithms \textbf{DGM} and \textbf{DFGM} we consider the
average of iterates $\hat{u}^k_\epsilon$. Since $Q \succeq
0$, we have chosen $\rho = {\cal O}(1/\epsilon)$. In the case of
Algorithm \textbf{DGM}, at each outer iteration the inner problem is
solved with accuracy $\ein = \epsilon$. For the Algorithm
\textbf{DFGM} we consider two scenarios: in the first one, the inner
problem is solved with accuracy $\ein = 0.001$, while in the second
one we use the theoretic inner accuracy \eqref{choose_ein}. We
observe a good behavior of Algorithm \textbf{DFGM},  comparable to
Cplex and Gurobi.
\begin{figure}[ht!]
\begin{center}
\vskip-0.1cm
\includegraphics[width=0.55\textwidth,height=5cm]{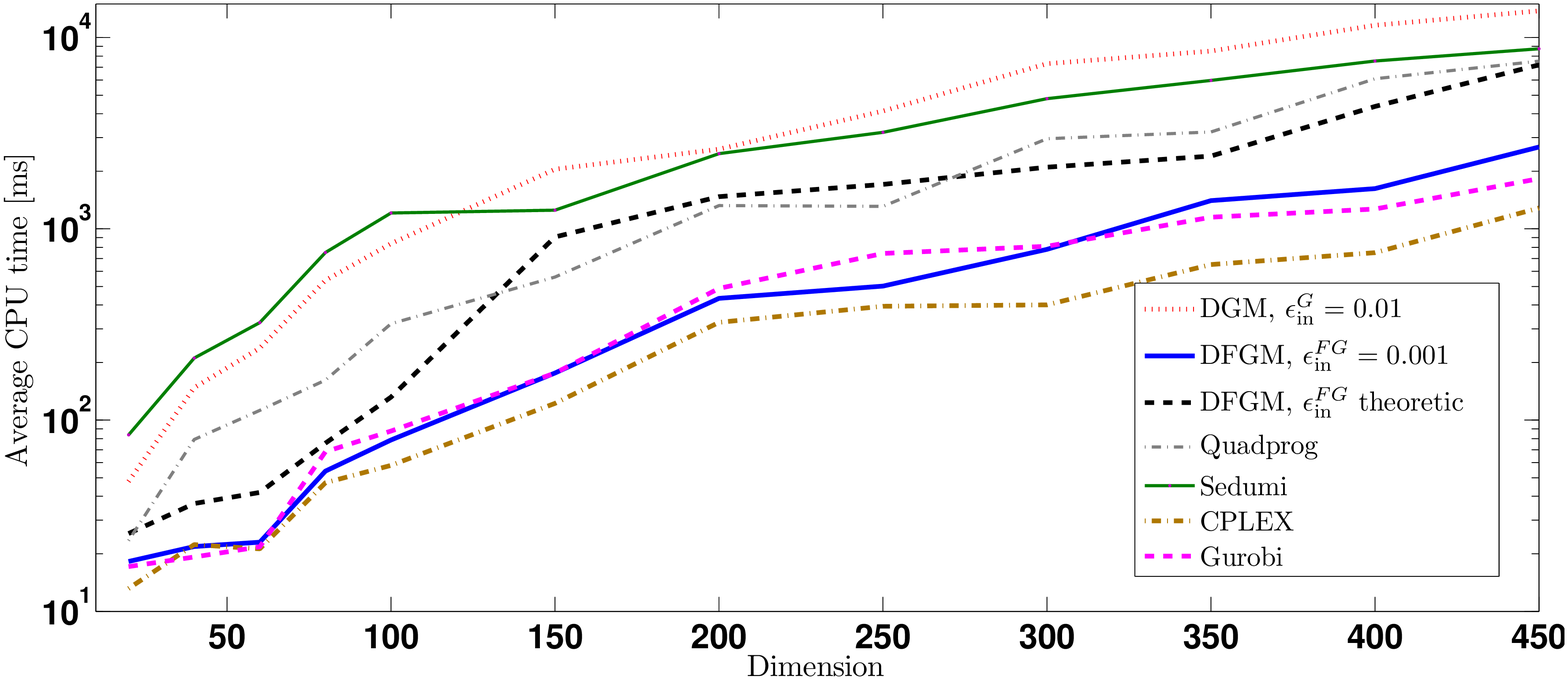}
\caption{Average CPU time (ms) for solving QPs, with $Q \succeq
0$ and $\calK=\{0\}$ of different dimensions, with several solvers.}
\label{fig:qp_cpu}
\end{center}
\end{figure}

%%%%%%%%%%%%%%%%%%%%%%%%%%%%%%%%%%%%%%%%%%%%%%%%%%%%%

\subsection{Numerical tests: case $\calK=\rset^p_{-}$}
\noindent We plot in Fig. \ref{fig:comparison_dfo} the number of
iterations of Algorithms \textbf{DGM} and  \textbf{DFGM} in the
primal last and  average iterates for $25$ random QPs with
inequality constraints ($Q \succ 0$ and $\calK=\rset^p_{-}$) of
variable dimension ranging from $n=10$ to $n = 500$. We choose the
accuracy $\epsilon=0.01$ and the stopping criteria was $\abs{F(u)
- F^{*}}$ and $\text{dist}_{\calK}(G u + g)$ less than the
accuracy $\epsilon$. From this figure we observe that the number of
iterations are not varying much for different test cases and also
that the number of iterations are mildly dependent on problem's
dimension. Finally, we  observe that dual first order  methods perform
usually better in the primal last iterate than in the average of primal  iterates.

\begin{figure}[ht!]
\begin{center}
\vskip-0.1cm
\includegraphics[width=1.05\textwidth,height=5cm]{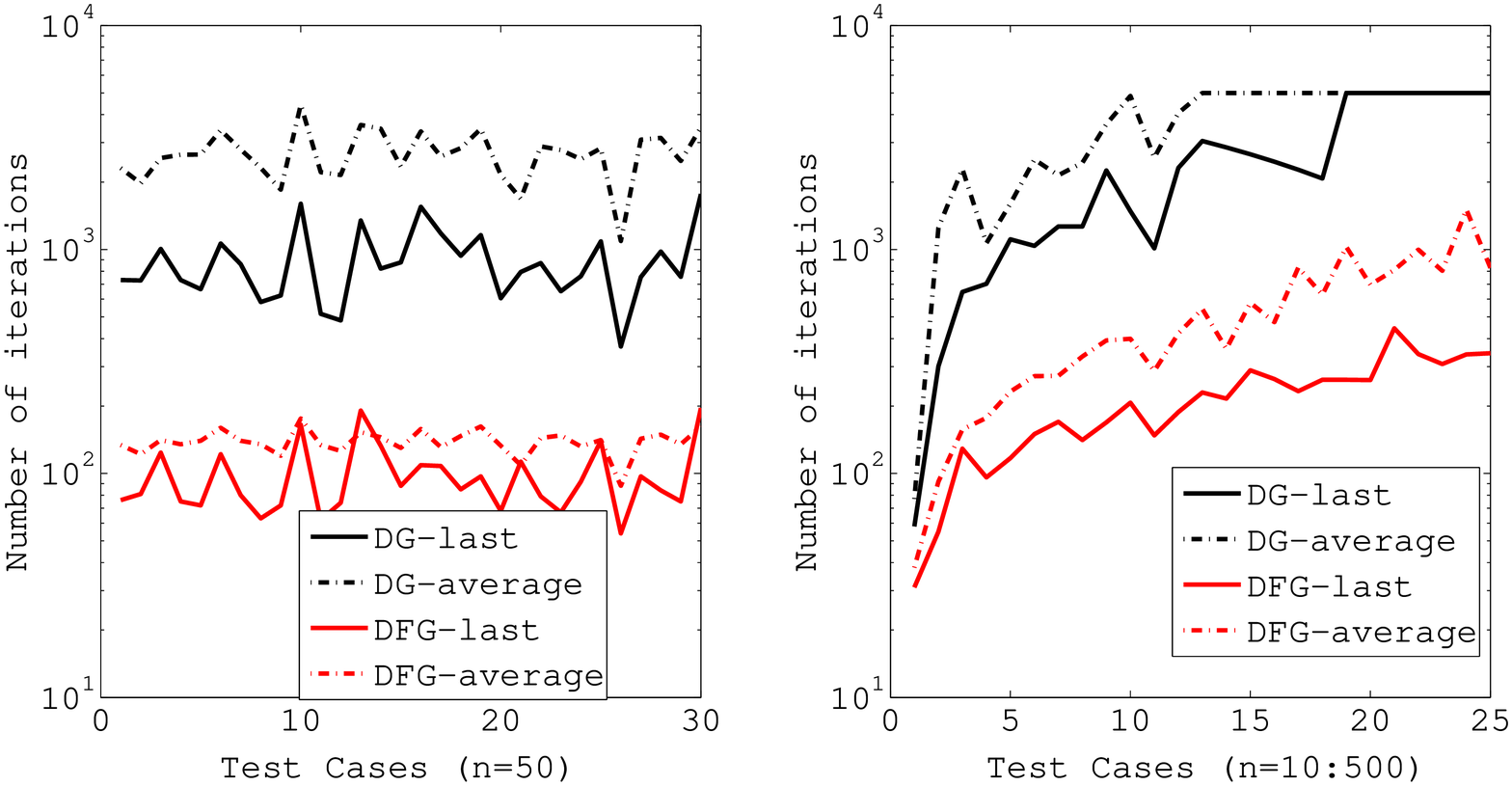}
\caption{Number of outer iterations on random QPs ($Q \succ 0$,
$\calK=\rset^p_{-}$) for  \textbf{DGM} and \textbf{DFGM} in primal last/average of iterates for different test cases of the same dimension (left) and  of variable dimension~(right).
} \label{fig:comparison_dfo}
\end{center}
\end{figure}

%%%%%%%%%%%%%%%%%%%%%%%%%%%%%%%%%%%%%%%%%%%%%%%%%%%%%%%%%%%%%%%%%%%%%%

\end{document}